\documentclass[12pt,oneside,reqno]{amsart}
\usepackage{array}
\usepackage{amssymb,tikz,amsmath,bm,color,graphicx,hyperref,anysize}
\usepackage{multicol}
\usepackage{mathtools}
\usetikzlibrary{circuits.logic.US,circuits.logic.IEC,fit}

\usetikzlibrary{cd}
\usepackage{enumitem}
\usepackage[numbers,sort&compress]{natbib}
\usepackage{comment}
\usetikzlibrary{arrows.meta, decorations.markings}
\theoremstyle{plain}
\newtheorem{theorem}{Theorem}[section]
\newtheorem{lemma}[theorem]{Lemma}
\newtheorem{corollary}[theorem]{Corollary}

\newenvironment{solution*}{%
  \par\noindent\textbf{Solution. }\normalfont
}{\qed}
\theoremstyle{definition}
\newtheorem{example}{Example}[section]
\newtheorem{definition}[theorem]{Definition}
\theoremstyle{remark}
\newtheorem{remark}[theorem]{Remark}
\newtheorem{question}[]{\textbf{Question}}
\newtheorem{answer}{\textbf{Answer}}
\newtheorem{observation}[]{\textbf{Observation}}
\newtheorem{algorithm}{Algorithm}
\usetikzlibrary{3d}
\usetikzlibrary{shapes.geometric, arrows}
\tikzstyle{startstop} = [rectangle, rounded corners, minimum width=3cm, minimum height=1cm,text centered, draw=black, fill=red!20]
\tikzstyle{process} = [rectangle, minimum width=3cm, minimum height=1cm, text centered, draw=black, fill=blue!10]
\tikzstyle{decision} = [diamond, minimum width=3cm, minimum height=1cm, text centered, draw=black, fill=green!20]
\tikzstyle{arrow} = [thick,->,>=stealth]
\newcommand{\mbb}[1]{\mathbb{#1}}

\newcommand{\mrm}[1]{\mathrm{#1}}
\newcommand{\btext}[1]{\mathbf{#1}}

\makeatletter
\def\@setcopyright{}
\def\serieslogo@{}
\makeatother
\begin{document}
\marginsize{2.54 cm}{2.54 cm}{2.54 cm}{2.54 cm}
% First we specify the top matter (author, title, etc).
% Note: All of the top matter items are optional and can be omitted.
% But you will probably want to specify at least the author and title
% and perhaps an abstract.
% author information
% first author
\author{Sanjay Mishra}
\address{Department of Mathematics \newline \indent Amity School of Applied Sciences, Amity University Lucknow Campus, UP, India}
\email{drsmishraresearch@gmail.com}
%\author{Suhaila Barekzai}
%\address{Department of Mathematics \newline \indent Lovely Professional University, Phagwara-144411, Punjab, India}
%\email{s.barekzai1995@gmail.com}
% title
\title[Simplicial Homology Groups]{Simplicial Homology Groups}
%%%%%%%%%%%%%%%%%%%%%%%%%%%%%%%%%%%%%%%%%%%%%%%%%%%%%%%%%%%%%%%%%%%%%%%%%%%%%%%%%%%%%%%%%%%%%%%%%%%%%%%%%%%%%%%%%
%                                                  Abstract
%%%%%%%%%%%%%%%%%%%%%%%%%%%%%%%%%%%%%%%%%%%%%%%%%%%%%%%%%%%%%%%%%%%%%%%%%%%%%%%%%%%%%%%%%%%%%%%%%%%%%%%%%%%%%%%%%
\begin{abstract}
This expository article presents a self-contained introduction to simplicial homology for finite simplicial complexes, emphasizing concrete computation and geometric intuition. Beginning with orientations of simplices and the construction of free abelian chain groups, the boundary operators are defined via the alternating-sum formula and shown to satisfy the chain-complex identity that the boundary of a boundary vanishes. Cycles and boundaries are then developed as kernels and images of the boundary maps, leading to homology groups that capture connected components, independent loops, and higher-dimensional voids. Throughout, detailed low-dimensional examples and step-by-step matrix calculations illustrate how to form boundary matrices, compute kernels and images, and identify generators and relations in \(H_p\). The presentation highlights universal properties of chain groups, clarifies sign conventions and induced orientations, and demonstrates the invariance of homology under combinatorial refinements, thereby connecting geometric features of spaces to computable algebraic invariants.
\end{abstract}

%%%%%%%%%%%%%%%%%%%%%%%%%%%%%%%%%%%%%%%%%%%%%%%%%%%%%%%%%%%%%%%%%%%%%%%%%%%%%%%%%%%%%%%%%%%%%%%%%%
% AMS subject classifications (used in AMS journals)
%\subjclass [2020] {see MSC Classifications}
\subjclass[2020] {Primary 55N10; Secondary 55U10, 55U15, 55-01.}
% AMS keywords (used in AMS journals)
\keywords{Simplicial complexes; Oriented simplices; Chain complexes; Boundary operator; Homology groups.}

% acknowledge support, etc
%\thanks{This research was partially supported by Lovely Professional University}
%\thanks{We would like to thank our colleagues for their helpful criticism.}
% dedication
%\dedicatory{ABC}
% today's date, or fill in whatever date you prefer
%\date{\today}
% This ends the top matter information.
% We can now tell LaTeX to display the top matter.
\maketitle
%\tableofcontents % Print the table of contents itself
% Having displayed the top matter, we now proceed to the body of the article.
\tableofcontents % Print the table of contents itself
%*****************************************************************************************************************
%                                         Introduction
%*****************************************************************************************************************
\section{Introduction}\label{s:Introduction}

Simplicial homology furnishes a sequence of abelian groups that quantify the connected components, loops, and higher-dimensional voids of a space modeled by a simplicial complex, in a way that is invariant under combinatorial refinements of the triangulation. In this article, the development is self-contained and example-driven: starting from orientations of simplices to control signs, passing through chain groups and boundary operators, and culminating in the construction and computation of homology groups as stable algebraic invariants.

The need for orientation appears at the outset. Assigning a coherent orientation to each simplex ensures that boundary computations respect cancellation on shared faces, making precise the principle that the boundary of a boundary vanishes. With orientations fixed, the group \(C_p(K)\) of \(p\)-chains is introduced as the free abelian group generated by the oriented \(p\)-simplices of a simplicial complex \(K\). This linear-algebraic viewpoint is emphasized throughout, preparing the way for explicit boundary matrices and rank computations.

The boundary operator \(\partial_p \colon C_p(K)\to C_{p-1}(K)\) is defined on elementary simplices by the alternating-sum formula and extended linearly to all chains. Worked examples in low dimensions illustrate induced orientations on faces and clarify sign conventions. A direct verification that \(\partial_{p-1}\circ \partial_p = 0\) establishes a chain complex, and a universal mapping property for chain groups highlights how data specified on generators extend uniquely to homomorphisms.

Cycles and boundaries arise naturally as \(Z_p(K) = \ker \partial_p\) and \(B_p(K)=\operatorname{im}\partial_{p+1}\), with \(B_p(K)\subseteq Z_p(K)\) yielding the \(p\)-th homology group
\[
H_p(K) \;=\; Z_p(K)\big/ B_p(K).
\]
These groups connect algebra to geometry: \(H_0\) detects connected components, \(H_1\) records independent loops that do not bound filled \(2\)-chains, \(H_2\) captures closed surfaces that do not bound \(3\)-chains, and higher groups continue this pattern in dimension. Short, concrete computations make these interpretations explicit.

On the computational side, the workflow is systematic: choose orientations, form chain groups, build boundary matrices, compute kernels and images, and take quotients to obtain \(H_p(K)\). Examples show when \(H_1\cong \mathbb{Z}\) (as for a single essential loop) and when it vanishes (as for a filled triangle), and how comparing \(\mathrm{rank}(\ker \partial_1)\) with \(\mathrm{rank}(\operatorname{im}\partial_2)\) resolves ambiguous cases. Throughout, the exposition balances formal definitions with step-by-step calculations, demystifying sign conventions, basis choices, and matrix manipulations that often obscure first encounters with simplicial homology.

Finally, while subdivisions change the combinatorics of a complex, the resulting homology remains unchanged, reflecting the underlying space rather than any specific triangulation. This robustness is the central theme: homology translates geometric features into computable algebra, preserving essential structure across equivalent models of the same space.

% Preamble example:
% \usepackage[authoryear]{natbib}
% \bibliographystyle{plainnat}

\section{Oriented Simplex}\label{s:Oriented Simplex}
The primary objective of introducing an oriented simplex is to assign a consistent direction or ordering to the vertices of a simplex, which is essential for defining boundary operators in homology theory. This orientation allows us to handle sign conventions systematically when computing boundaries, ensuring that shared faces between adjacent simplices cancel appropriately in a chain. Without orientation, we couldn’t define boundaries in a way that preserves algebraic structure, making orientation a crucial concept for the development of homology groups.

\begin{definition}[Permutation of the Vertices of a Simplex]\label{d:Permutation of the Vertices of a Simplex}
Let $\sigma_p$ be a $p$-dimensional simplex with ordered vertex set
$V = \{\mathbf{a}_0, \mathbf{a}_1, \ldots, \mathbf{a}_p\}$. A permutation of the vertices is a bijective function
\[
f \colon \{0, 1, \ldots, p\} \to \{0, 1, \ldots, p\},
\]
which gives a new ordering of the vertices as
\[
(\mathbf{a}_{f(0)}, \mathbf{a}_{f(1)}, \ldots, \mathbf{a}_{f(p)}).
\]
\end{definition}

\begin{observation}
The expression
\[
(\mathbf{a}_{f(0)}, \mathbf{a}_{f(1)}, \ldots, \mathbf{a}_{f(p)})
\]
represents the reordered list of vertices of the simplex after applying the permutation $f$ to the original ordered vertex set.

Here’s a breakdown:
\begin{itemize}
    \item Initially, the vertices of the $p$-dimensional simplex are ordered as
    \[
    \{\mathbf{a}_0, \mathbf{a}_1, \ldots, \mathbf{a}_p\}.
    \]
    \item The permutation $f$ is a function that rearranges these indices. For example, if $f$ sends $0 \to 2$, $1 \to 0$, and $2 \to 1$, then the permutation will reorder the vertices as
    \[
    (\mathbf{a}_2, \mathbf{a}_0, \mathbf{a}_1).
    \]
    \item In this case:
    \begin{itemize}
        \item $\mathbf{a}_{f(0)}$ means the vertex corresponding to index $0$ after applying $f$,
        \item $\mathbf{a}_{f(1)}$ means the vertex corresponding to index $1$ after applying $f$, and so on.
    \end{itemize}
\end{itemize}

Thus, the tuple
\[
(\mathbf{a}_{f(0)}, \mathbf{a}_{f(1)}, \ldots, \mathbf{a}_{f(p)})
\]
represents the new order of the vertices after applying the permutation $f$ to the original ordering.
\end{observation}

\begin{example}
Consider 2-simplex $\sigma_2$ (a triangle) and its vertices labeled as $\mathbf{a}_0, \mathbf{a}_1, \mathbf{a}_2$. So, the initial ordering of the vertices is $V = \{\mathbf{a}_0, \mathbf{a}_1, \mathbf{a}_2\}$. Now, consider a permutation $\sigma$ of the vertices. For instance, let’s say:
\[
f \colon \{0, 1, 2\} \to \{0, 1, 2\}, \text{defined by} \, f(0) = 1, f(1) = 2, f(2) = 0\]
This means that the permutation will reorder the vertices as:
\[
(\mathbf{a}_{f(0)}, \mathbf{a}_{f(1)}, \mathbf{a}_{f(2)}) = (\mathbf{a}_1, \mathbf{a}_2, \mathbf{a}_0).
\]
So, after applying the permutation $f$, the vertices of the simplex are now in the order $\mathbf{a}_1, \mathbf{a}_2, \mathbf{a}_0$, which is a rearranged ordering compared to the original $(\mathbf{a}_0, \mathbf{a}_1, \mathbf{a}_2)$.
\end{example}

\begin{definition}[Even and Odd Permutations of the Vertex Set]\label{d:Even and Odd Permutations of the Vertex Set}
Let $\sigma_p$ be a $p$-dimensional simplex with ordered vertex set $V = \{\mathbf{a}_0, \mathbf{a}_1, \ldots, \mathbf{a}_p\}$.
Let $f \colon \{0, 1, \ldots, p\} \to \{0, 1, \ldots, p\}$ be a permutation function that rearranges the vertices as  $(\mathbf{a}_{f(0)}, \mathbf{a}_{f(1)}, \ldots, \mathbf{a}_{f(p)})$. Then:
\begin{itemize}
\item The permutation $f$ is called \textbf{even} if it can be expressed as a composition of an even number of transpositions (i.e., swaps of two vertices).
\item The permutation $f$ is called \textbf{odd} if it can be expressed as a composition of an odd number of transpositions.
\end{itemize}

The \textbf{sign} of the permutation $f$, denoted $\operatorname{sgn}(f)$, is defined as:
\[
\operatorname{sgn}(f) =
\begin{cases}
+1, & \text{if } f \text{ is even,} \\
-1, & \text{if } f \text{ is odd.}
\end{cases}
\]
\end{definition}

\begin{example}\label{eg:Even and Odd Permutations on a 2-Simplex}
Let $\sigma_2$ be a 2-simplex with vertex set $V = \{\mathbf{a}_0, \mathbf{a}_1, \mathbf{a}_2\}$. Let natural ordering is $(\mathbf{a}_0, \mathbf{a}_1, \mathbf{a}_2)$.

\begin{enumerate}
\item \textbf{Even Permutation:} Let $f$ be the permutation on $V$ defined by 
\[
f(0) = 1, \quad f(1) = 2, \quad f(2) = 0,
\]
which gives,
\[
(\mathbf{a}_{f(0)}, \mathbf{a}_{f(1)}, \mathbf{a}_{f(2)}) = (\mathbf{a}_1, \mathbf{a}_2, \mathbf{a}_0).
\]
We can write this as a composition of two transpositions $f = (0\ 2)(0\ 1)$, because
\begin{itemize}
    \item Swap $\mathbf{a}_0 \leftrightarrow \mathbf{a}_1$ $\Rightarrow$ $(\mathbf{a}_1, \mathbf{a}_0, \mathbf{a}_2)$
    \item Then swap $\mathbf{a}_0 \leftrightarrow \mathbf{a}_2$ $\Rightarrow$ $(\mathbf{a}_1, \mathbf{a}_2, \mathbf{a}_0)$
\end{itemize}
Since it takes two transpositions (even), $f$ is an even permutation. 
Hence, $\operatorname{sgn}(f) = +1$.
\item \textbf{Odd Permutation:} Let $g$ be the permutation on $V$ defined by 
\[
g(0) = 1, \quad g(1) = 0, \quad g(2) = 2,
\]
which gives,
\[
(\mathbf{a}_{g(0)}, \mathbf{a}_{g(1)}, \mathbf{a}_{g(2)}) = (\mathbf{a}_1, \mathbf{a}_0, \mathbf{a}_2).
\]
This is a single transposition $g = (0\ 1)$, since we just swap $\mathbf{a}_0 \leftrightarrow \mathbf{a}_1$. One transposition (odd) implies $g$ is an odd permutation. So, $\operatorname{sgn}(g) = -1$.
\end{enumerate}
\end{example}

\begin{algorithm}[Algorithm to Check if a Permutation is Even or Odd]
Let the permutation $f$ be given in list form 
\[
f = [f(0), f(1), \ldots, f(p)].
\]    
Apply following steps to check $f$ is even or odd. 
\begin{enumerate}
 \item \textbf{Count all inversions} in the list.  
    An \textbf{inversion} is a pair \( (i, j) \) such that:
    \[
    i < j \quad \text{and} \quad f(i) > f(j).
    \] 
 \item Let \( N \) be the total number of inversions.

    \item Then:
    \begin{itemize}
        \item If \( N \) is \textbf{even}, the permutation is \textbf{even}.
        \item If \( N \) is \textbf{odd}, the permutation is \textbf{odd}.
    \end{itemize}
\end{enumerate}
\end{algorithm}

\begin{example}
Let the permutation $f$ be given in list form 
\[
f = [f(0) = 2, f(1) = 0, f(2) = 1]
\]    
We count inversions:
\begin{itemize}
    \item \( f(0) = 2 > f(1) = 0 \Rightarrow \text{inversion} \)
    \item \( f(0) = 2 > f(2) = 1 \Rightarrow \text{inversion} \)
    \item \( f(1) = 0 < f(2) = 1 \Rightarrow \text{no inversion}\)
\end{itemize}
Since total number of inversion is 2, therefore the permutation $f$ is even. 
\end{example}

\begin{definition}[Equivalence of Ordering of Vertex]\label{d:Equivalence of Ordering of Vertex}
Two orderings of the vertex set  $V = \{\mathbf{a}_0, \mathbf{a}_1, \ldots, \mathbf{a}_p\}$
are said to be equivalent if the permutation that transforms one ordering into the other is an even permutation. In other words, two orderings are equivalent if one can be obtained from the other by applying an even number of transpositions (swaps of two vertices).
\end{definition}

\begin{example}\label{eg:Equivalence of Ordering of Vertices}
Let the vertex set of a 2-simplex be \( V = \{\mathbf{a}_0, \mathbf{a}_1, \mathbf{a}_2\} \). Consider two different orderings of the vertices. The first ordering is \( (\mathbf{a}_0, \mathbf{a}_1, \mathbf{a}_2) \), and the second ordering is \( (\mathbf{a}_1, \mathbf{a}_2, \mathbf{a}_0) \). The permutation \( f \) that maps the first to the second is given by \( f(0) = 1, f(1) = 2, f(2) = 0 \), which corresponds to the cycle \( (0\ 1\ 2) \). This cycle can be written as a composition of two transpositions: \( (0\ 2)(0\ 1) \), which is an even number of transpositions. Therefore, the second ordering is \textbf{equivalent} to the first.

Now consider another ordering \( (\mathbf{a}_1, \mathbf{a}_0, \mathbf{a}_2) \). The permutation here is \( f(0) = 1, f(1) = 0, f(2) = 2 \), which corresponds to the transposition \( (0\ 1) \), a single (odd) transposition. Hence, this ordering is \textbf{not equivalent} to the original ordering \( (\mathbf{a}_0, \mathbf{a}_1, \mathbf{a}_2) \).
\end{example}

\begin{theorem}[Ordering as an Equivalence Relation]\label{t:Ordering as an Equivalence Relation}
Let $V = \{\mathbf{a}_0, \mathbf{a}_1, \ldots, \mathbf{a}_p\}$ be the vertex set of a $p$-simplex $\sigma_p$, and let $O$ denote the set of all possible orderings of the vertices of $V$. Define a relation $\sim$ on $O$ by:
\[
(o_1 \sim o_2) \iff \text{the permutation that takes } o_1 \text{ to } o_2 \text{ is even}.
\]

Then, $\sim$ is an equivalence relation on $O$, and the equivalence classes under this relation define the two possible orientations of the simplex $\sigma_p$.
\end{theorem}

\begin{definition}[Orientation of Simplex]\label{d:Orientation of Simplex}
Let $\sigma_p$ be a $p$-simplex with vertex set $V = \{\mathbf{a}_0, \mathbf{a}_1, \ldots, \mathbf{a}_p\}$. An  orientation of $\sigma_p$ is defined as an equivalence class of orderings of its vertices, where two orderings are considered equivalent if the permutation that transforms one into the other is an even permutation.

In simple terms, a $p$-simplex has exactly two orientations: one corresponding to the class of even permutations of a given ordering, and the other to the class of odd permutations.
\end{definition}
\begin{example}\label{eg:Orientation of a 2-Simplex}
Let the vertex set of a 2-simplex $\sigma_2$ be $V = \{\mathbf{a}_0, \mathbf{a}_1, \mathbf{a}_2\}$. The set of all orderings of these vertices includes the following six permutations:  
\[S = \{(\mathbf{a}_0, \mathbf{a}_1, \mathbf{a}_2), (\mathbf{a}_1, \mathbf{a}_2, \mathbf{a}_0),(\mathbf{a}_2, \mathbf{a}_0, \mathbf{a}_1), (\mathbf{a}_0, \mathbf{a}_2, \mathbf{a}_1), (\mathbf{a}_2, \mathbf{a}_1, \mathbf{a}_0), (\mathbf{a}_1, \mathbf{a}_0, \mathbf{a}_2)\}\]
Now, consider the sign of the permutation required to transform the standard ordering $(\mathbf{a}_0, \mathbf{a}_1, \mathbf{a}_2)$ into each of the others. The permutations
\[(\mathbf{a}_0, \mathbf{a}_2, \mathbf{a}_1), \quad
(\mathbf{a}_1, \mathbf{a}_2, \mathbf{a}_0),\quad (\mathbf{a}_2, \mathbf{a}_0, \mathbf{a}_1)
\]
are even permutations, each with $\operatorname{sgn}(f) = +1$. Thus, they belong to the same equivalence class as the standard ordering and represent the same orientation of the simplex.

On the other hand, the permutations
\[
(\mathbf{a}_0, \mathbf{a}_2, \mathbf{a}_1),\quad (\mathbf{a}_2, \mathbf{a}_1, \mathbf{a}_0),\quad (\mathbf{a}_1, \mathbf{a}_0, \mathbf{a}_2)
\]
are odd permutations, each with $\operatorname{sgn}(f) = -1$, and therefore belong to a different equivalence class, representing the opposite orientation.

Hence, the orientation of a 2-simplex is determined by the equivalence class of vertex orderings under even permutations. There are exactly two such classes, corresponding to the two possible orientations of the simplex as follows:
\begin{gather*}
E=
\{ (\mathbf{a}_0, \mathbf{a}_1, \mathbf{a}_2), (\mathbf{a}_1, \mathbf{a}_2, \mathbf{a}_0), (\mathbf{a}_2, \mathbf{a}_0, \mathbf{a}_1) \}\\ 
O =\{ (\mathbf{a}_0, \mathbf{a}_2, \mathbf{a}_1), (\mathbf{a}_2, \mathbf{a}_1, \mathbf{a}_0), (\mathbf{a}_1, \mathbf{a}_0, \mathbf{a}_2) \}
\end{gather*}
\end{example}

\begin{definition}[Oriented Simplex]\label{d:Oriented Simplex}
An \textbf{oriented simplex} is a simplex together with a choice of one of the two possible orientations of its ordered vertices. More formally, if $\sigma_p = (\mathbf{a}_0, \mathbf{a}_1, \ldots, \mathbf{a}_p)$ is a $p$-simplex with an ordered set of vertices, then the oriented simplex is denoted by:
\[
[\mathbf{a}_0, \mathbf{a}_1, \ldots, \mathbf{a}_p]
\]
where the square brackets indicate that the ordering is taken up to even permutations—that is, any ordering obtained by an even permutation is considered the same orientation, while an odd permutation gives the opposite orientation.
\end{definition}

\begin{example}[Oriented 1-Simplex]\label{eg:Oriented 1-Simplex}
Let the vertex set be $V = \{\mathbf{a}_0, \mathbf{a}_1\}$. There are two possible orderings of these vertices:
\begin{itemize}
  \item $(\mathbf{a}_0, \mathbf{a}_1)$ — this gives the \textbf{positive orientation}, written as:
  \[
  [\mathbf{a}_0, \mathbf{a}_1]
  \]
  \item $(\mathbf{a}_1, \mathbf{a}_0)$ — this is the \textbf{reverse order}, and hence gives the \textbf{opposite orientation}, written as:
  \[
  -[\mathbf{a}_0, \mathbf{a}_1]
  \]
\end{itemize}
Thus, both orderings define the same 1-simplex geometrically, but with opposite orientations.

\begin{figure}[h!]
\centering
\begin{tikzpicture}[scale=1.2, >=stealth]

  % Positive Orientation
  \node (A1) at (0,0) [circle,fill,inner sep=2pt,label=below:$\mathbf{a}_0$] {};
  \node (B1) at (3,0) [circle,fill,inner sep=2pt,label=below:$\mathbf{a}_1$] {};
  \draw[->, thick] (A1) -- (B1);
  \node at (1.5,0.4) {$[\mathbf{a}_0, \mathbf{a}_1]$};

  % Negative Orientation
  \node (A2) at (6,0) [circle,fill,inner sep=2pt,label=below:$\mathbf{a}_0$] {};
  \node (B2) at (9,0) [circle,fill,inner sep=2pt,label=below:$\mathbf{a}_1$] {};
  \draw[<-, thick] (A2) -- (B2);
  \node at (7.5,0.4) {$[\mathbf{a}_1, \mathbf{a}_0] = -[\mathbf{a}_0, \mathbf{a}_1]$};
\end{tikzpicture}
\caption{Illustration of a 1-simplex with two different orientations. The left shows the positively oriented simplex $[\mathbf{a}_0, \mathbf{a}_1]$, while the right shows the opposite orientation $[\mathbf{a}_1, \mathbf{a}_0] = -[\mathbf{a}_0, \mathbf{a}_1]$.}
\label{fig:oriented-1-simplex}
\end{figure}
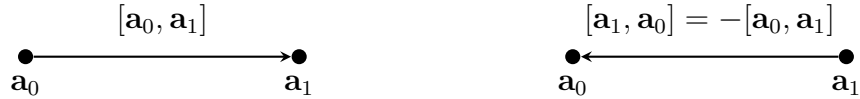
\end{example}

\begin{example}[Oriented 2-Simplex]\label{eg:Oriented 2-Simplex}
Let the vertex set be $V = \{\mathbf{a}_0, \mathbf{a}_1, \mathbf{a}_2\}$. Then the ordered simplex $\sigma_2 = (\mathbf{a}_0, \mathbf{a}_1, \mathbf{a}_2)$ represents one orientation of the 2-simplex, and we write the oriented simplex as:  
\[
[\mathbf{a}_0, \mathbf{a}_1, \mathbf{a}_2].
\]
This oriented simplex also includes the orderings:
\[
[\mathbf{a}_1, \mathbf{a}_2, \mathbf{a}_0] \quad \text{and} \quad [\mathbf{a}_2, \mathbf{a}_0, \mathbf{a}_1],
\]
because they are even permutations of the original ordering and hence define the same orientation.

On the other hand, the orderings
\[
(\mathbf{a}_0, \mathbf{a}_2, \mathbf{a}_1), \quad (\mathbf{a}_2, \mathbf{a}_1, \mathbf{a}_0), \quad (\mathbf{a}_1, \mathbf{a}_0, \mathbf{a}_2)
\]
represent the opposite orientation, and we write that as:
\[
-[\mathbf{a}_0, \mathbf{a}_1, \mathbf{a}_2].
\]
The negative sign (like in $-[\mathbf{a}_0, \mathbf{a}_1, \mathbf{a}_2]$) is a conventional way to indicate that the orientation is opposite to the positively oriented simplex $[\mathbf{a}_0, \mathbf{a}_1, \mathbf{a}_2]$. The positive orientation is counterclockwise while negative orientation is clockwise. 
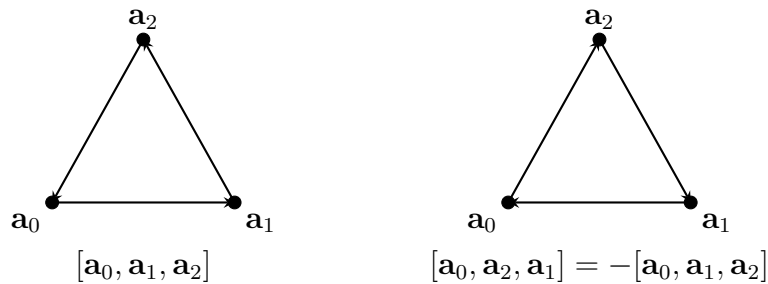
\begin{figure}[h!]
\centering
\begin{tikzpicture}[scale=1.2, >=stealth]

  % Positive orientation
  \coordinate (A1) at (0,0);
  \coordinate (B1) at (2,0);
  \coordinate (C1) at (1,1.8);
  \filldraw (A1) circle (2pt) node[below left] {$\mathbf{a}_0$};
  \filldraw (B1) circle (2pt) node[below right] {$\mathbf{a}_1$};
  \filldraw (C1) circle (2pt) node[above] {$\mathbf{a}_2$};

  \draw[->, thick] (A1) -- (B1);
  \draw[->, thick] (B1) -- (C1);
  \draw[->, thick] (C1) -- (A1);
  \node at (1, -0.7) {$[\mathbf{a}_0, \mathbf{a}_1, \mathbf{a}_2]$};

  % Negative orientation
  \coordinate (A2) at (5,0);
  \coordinate (B2) at (7,0);
  \coordinate (C2) at (6,1.8);
  \filldraw (A2) circle (2pt) node[below left] {$\mathbf{a}_0$};
  \filldraw (B2) circle (2pt) node[below right] {$\mathbf{a}_1$};
  \filldraw (C2) circle (2pt) node[above] {$\mathbf{a}_2$};

  \draw[<-, thick] (A2) -- (B2);
  \draw[<-, thick] (B2) -- (C2);
  \draw[<-, thick] (C2) -- (A2);
  \node at (6, -0.7) {$[\mathbf{a}_0, \mathbf{a}_2, \mathbf{a}_1] = -[\mathbf{a}_0, \mathbf{a}_1, \mathbf{a}_2]$};

\end{tikzpicture}
\caption{Illustration of a 2-simplex (triangle) with two orientations. Left: counterclockwise (positive) orientation $[\mathbf{a}_0, \mathbf{a}_1, \mathbf{a}_2]$. Right: clockwise (negative) orientation $[\mathbf{a}_0, \mathbf{a}_2, \mathbf{a}_1] = -[\mathbf{a}_0, \mathbf{a}_1, \mathbf{a}_2]$.}
\label{fig:oriented-2-simplex}
\end{figure}
\end{example}

\begin{example}[Oriented 3-simplex]\label{eg:Oriented 3-simplex}
Complete this.
\end{example}

\section{Group of Oriented $p$-chains}\label{s:Group of Oriented p chains}
\begin{definition}[$p$-chain on Simplicial Complex]\label{d:p chain on Simplicial Complex}
Let $K$ is simplicial complex. A $p$-chain on $K$ is a function $c$ from the set of oriented $p$-simplices of $K$ to $\mbb{Z}$  ($c\colon \widetilde{K}_p \longrightarrow \mathbb{Z}$ where $\widetilde{K}_p$ is the set of oriented $p$-simplices of $K$), such that 
\begin{enumerate}
\item $c(\sigma) = -c(\sigma')$, if $\sigma$ and $\sigma'$ are opposite orientation of the same simplex.
\item $c(\sigma) = 0$ for all but finitely many oriented $p$-simplices $\sigma$.
\end{enumerate}    
\end{definition}

\begin{example}[$p$-chain on Simplicial Complex]\label{eg:p chain on Simplicial Complex}
Let $K$ be the simplicial complex corresponding to a tetrahedron with vertex set $V = \{\mathbf{a}_0, \mathbf{a}_1, \mathbf{a}_2, \mathbf{a}_3\}$. The simplicial complex $K$ is given as
\[
K = \left\{ \sigma \subseteq V \colon  \sigma \text{ is a non-empty subset of } \{\mathbf{a}_0, \mathbf{a}_1, \mathbf{a}_2, \mathbf{a}_3\} \, \text{with} \, o(\sigma) \leq 4 \right\}
\]
Explicitly, the elements of $K$ are
\[
\begin{aligned}
K = \{ & \{\mathbf{a}_0\}, \{\mathbf{a}_1\}, \{\mathbf{a}_2\}, \{\mathbf{a}_3\}, \\
& \{\mathbf{a}_0, \mathbf{a}_1\}, \{\mathbf{a}_0, \mathbf{a}_2\}, \{\mathbf{a}_0, \mathbf{a}_3\}, \{\mathbf{a}_1, \mathbf{a}_2\}, \{\mathbf{a}_1, \mathbf{a}_3\}, \{\mathbf{a}_2, \mathbf{a}_3\}, \\
& \{\mathbf{a}_0, \mathbf{a}_1, \mathbf{a}_2\}, \{\mathbf{a}_0, \mathbf{a}_1, \mathbf{a}_3\}, \{\mathbf{a}_0, \mathbf{a}_2, \mathbf{a}_3\}, \{\mathbf{a}_1, \mathbf{a}_2, \mathbf{a}_3\}, \\
& \{\mathbf{a}_0, \mathbf{a}_1, \mathbf{a}_2, \mathbf{a}_3\} \}.
\end{aligned}
\]
We now describe the $p$-chains $C_p(K)$ for $p = 0, 1, 2, 3$, using the notation $\widetilde{K}_p$ to denote the set of oriented $p$-simplices of $K$. Each $p$-chain is a function from $\widetilde{K}_p$ to $\mathbb{Z}$ with finite support, and is expressed as a finite formal sum of oriented $p$-simplices with integer coefficients.
\begin{enumerate}
\item \textbf{0-chains:}The set of oriented $0$-simplices is
\[
\widetilde{K}_0 = \{ [\mathbf{a}_0], [\mathbf{a}_1], [\mathbf{a}_2], [\mathbf{a}_3] \}.
\]
A general $0$-chain on $K$ is
\[
c^{(0)} = a_0 [\mathbf{a}_0] + a_1 [\mathbf{a}_1] + a_2 [\mathbf{a}_2] + a_3 [\mathbf{a}_3], \quad \text{where } a_i \in \mathbb{Z}.
\]
\textbf{Example:}Let $c^{(0)}$ be defined on $\widetilde{K}_0$ by
\[
c^{(0)} = 2[\mathbf{a}_0] - 1[\mathbf{a}_1] + 0[\mathbf{a}_2] + 3[\mathbf{a}_3].
\]
This 0-chain assigns integer values to each vertex. A $0$-chain is a formal linear combination of the vertices of the tetrahedron with integer coefficients. Since orientation does not matter for $0$-simplices, the 0-chains simply represent integer-valued weightings on the vertices. In our example, we assign the integer $2$ to vertex $\mathbf{a}_0$, $-1$ to $\mathbf{a}_1$, $0$ to $\mathbf{a}_2$, and $3$ to $\mathbf{a}_3$. 
\item \textbf{1-chains:}The set of oriented $1$-simplices is
\[
\widetilde{K}_1 = \{ [\mathbf{a}_0,\mathbf{a}_1], [\mathbf{a}_0,\mathbf{a}_2], [\mathbf{a}_0,\mathbf{a}_3], [\mathbf{a}_1,\mathbf{a}_2], [\mathbf{a}_1,\mathbf{a}_3], [\mathbf{a}_2,\mathbf{a}_3] \}
\]
A general $1$-chain on $K$ is
\begin{align*}
c^{(1)} & =   b_1[\mathbf{a}_0,\mathbf{a}_1] + b_2[\mathbf{a}_0,\mathbf{a}_2] + b_3[\mathbf{a}_0,\mathbf{a}_3] \\
& + b_4[\mathbf{a}_1,\mathbf{a}_2] + b_5[\mathbf{a}_1,\mathbf{a}_3] + b_6[\mathbf{a}_2,\mathbf{a}_3], \, b_i \in \mathbb{Z}        
\end{align*}
\textbf{Example:}Let $c^{(1)}$ be defined on $\widetilde{K}_1$ by
\[
c^{(1)} = [\mathbf{a}_0,\mathbf{a}_1] - 2[\mathbf{a}_1,\mathbf{a}_2] + [\mathbf{a}_2,\mathbf{a}_3].
\]
This assigns weights to oriented edges, respecting orientation. A $1$-chain is a formal sum of oriented edges (1-simplices), with orientation taken into account. For example, in the $1$-chain the coefficient $1$ on $[\mathbf{a}_0,\mathbf{a}_1]$ indicates one unit along that edge in the direction from $\mathbf{a}_0$ to $\mathbf{a}_1$. The $-2$ on $[\mathbf{a}_1,\mathbf{a}_2]$ means two units in the direction from $\mathbf{a}_2$ to $\mathbf{a}_1$ (opposite orientation), and $1$ on $[\mathbf{a}_2,\mathbf{a}_3]$ is a unit from $\mathbf{a}_2$ to $\mathbf{a}_3$. 
\item \textbf{2-chains:}The set of oriented $2$-simplices is
\[
\widetilde{K}_2 = \{ [\mathbf{a}_0,\mathbf{a}_1,\mathbf{a}_2], [\mathbf{a}_0,\mathbf{a}_1,\mathbf{a}_3], [\mathbf{a}_0,\mathbf{a}_2,\mathbf{a}_3], [\mathbf{a}_1,\mathbf{a}_2,\mathbf{a}_3] \}.
\]
A general $2$-chain on $K$ is
\[
c^{(2)} = c_1[\mathbf{a}_0,\mathbf{a}_1,\mathbf{a}_2] + c_2[\mathbf{a}_0,\mathbf{a}_1,\mathbf{a}_3] + c_3[\mathbf{a}_0,\mathbf{a}_2,\mathbf{a}_3] + c_4[\mathbf{a}_1,\mathbf{a}_2,\mathbf{a}_3], \quad c_i \in \mathbb{Z}.
\]
\textbf{Example:}Let $c^{(2)}$ be defined on $\widetilde{K}_2$ by
\[
c^{(2)} = [\mathbf{a}_0,\mathbf{a}_1,\mathbf{a}_2] - [\mathbf{a}_0,\mathbf{a}_2,\mathbf{a}_3].
\]
A $2$-chain is a formal combination of oriented 2-simplices (triangular faces of the tetrahedron), again with orientation considered. This example represents a sum of two oriented triangles. The first term, $[\mathbf{a}_0,\mathbf{a}_1,\mathbf{a}_2]$, is the face formed by those three vertices in counterclockwise order (or some fixed orientation). The second term is a negatively oriented copy of the face $[\mathbf{a}_0,\mathbf{a}_2,\mathbf{a}_3]$, meaning we traverse the triangle in the opposite direction. 
\item \textbf{3-chains:}The set of oriented $3$-simplices is
\[
\widetilde{K}_3 = \{ [\mathbf{a}_0,\mathbf{a}_1,\mathbf{a}_2,\mathbf{a}_3] \}.
\]
A general $3$-chain on $K$ is
\[
c^{(3)} = d \cdot [\mathbf{a}_0,\mathbf{a}_1,\mathbf{a}_2,\mathbf{a}_3], \quad d \in \mathbb{Z}.
\]
\textbf{Example:}Let $c^{(3)}$ be defined on $\widetilde{K}_3$ by 
\[
c^{(3)} = -1 [\mathbf{a}_0,\mathbf{a}_1,\mathbf{a}_2,\mathbf{a}_3].
\]
The $3$-chains on a tetrahedron are combinations of the full 3-simplex. Since there is only one 3-simplex in the tetrahedron, the set $\widetilde{K}_3$ has only one element: $[\mathbf{a}_0,\mathbf{a}_1,\mathbf{a}_2,\mathbf{a}_3]$. This example assigns the value $-1$ to the single 3-simplex, indicating that it is taken with the opposite orientation. 
\end{enumerate}
\textbf{Note:}In all cases, the coefficients are integers, and orientation matters. Reversing orientation changes the sign of the simplex. Each $p$-chain is a finitely supported function $c \colon \widetilde{K}_p \to \mathbb{Z}$ such that $c(\bar{\sigma}) = -c(\sigma)$ for opposite orientations.
\end{example}

\begin{definition}[$p$-chain group]\label{def:p-chain-group}
Let $K$ be a simplicial complex, and let $\widetilde{K}_p$ denote the set of oriented $p$-simplices of $K$. The \emph{$p$-chain group} of $K$, denoted by $C_p(K)$, is the free abelian group generated by the oriented $p$-simplices in $K$. That is,
\[
C_p(K) = \left\{ c : \widetilde{K}_p \to \mathbb{Z} \;\middle|\;
\begin{aligned}
&c(\sigma) = -c(\sigma') \text{ if } \sigma \text{ and } \sigma' \text{ have opposite orientations,} \\\\
&c(\sigma) = 0 \text{ for all but finitely many } \sigma \in \widetilde{K}_p
\end{aligned}
\right\}.
\]
\end{definition}

\begin{definition}[Group Operation on $C_p(K)$]
Let $K$ be a simplicial complex and let $C_p(K)$ denote the $p$-chain group of $K$. For any two $p$-chains $c_1, c_2 \in C_p(K)$, the group operation $+$ is defined pointwise as follows:
\[
(c_1 + c_2)(\sigma) = c_1(\sigma) + c_2(\sigma) \quad \text{for all } \sigma \in \widetilde{K}_p,
\]
where $\widetilde{K}_p$ is the set of oriented $p$-simplices of $K$.

The zero element of $C_p(K)$ is the function $0$ such that $0(\sigma) = 0$ for all $\sigma \in \widetilde{K}_p$. The inverse of a $p$-chain $c$ is the chain $-c$ defined by
\[
(-c)(\sigma) = -c(\sigma) \quad \text{for all } \sigma \in \widetilde{K}_p.
\]
\end{definition}

\begin{theorem}\label{thm:CpK-is-group}
Let $K$ be a simplicial complex and let $C_p(K)$ be the set of all $p$-chains on $K$. Then $C_p(K)$ forms an abelian group under the operation of pointwise addition defined by
\[
(c_1 + c_2)(\sigma) = c_1(\sigma) + c_2(\sigma), \quad \text{for all } \sigma \in \widetilde{K}_p,
\]
where $\widetilde{K}_p$ is the set of oriented $p$-simplices of $K$.
\end{theorem}

\begin{proof}
We will verify that $C_p(K)$ satisfies all the axioms of an abelian group under the operation of pointwise addition.
\begin{enumerate}
\item  \textbf{Closure:} Let $c_1, c_2 \in C_p(K)$. Then for each $\sigma \in \widetilde{K}_p$, the function $c_1 + c_2$ is defined by
\[
(c_1 + c_2)(\sigma) = c_1(\sigma) + c_2(\sigma).
\]
Since $c_1$ and $c_2$ are integer-valued and nonzero for only finitely many simplices, the same holds for $c_1 + c_2$. Also, since $c_i(\sigma') = -c_i(\sigma)$ whenever $\sigma'$ is the opposite orientation of $\sigma$, it follows that:
\[
(c_1 + c_2)(\sigma') = c_1(\sigma') + c_2(\sigma') = -c_1(\sigma) - c_2(\sigma) = -(c_1 + c_2)(\sigma).
\]
Thus, $c_1 + c_2$ is a $p$-chain, so $c_1 + c_2 \in C_p(K)$.
\item \textbf{Associativity:} Let $c_1, c_2, c_3 \in C_p(K)$. For all $\sigma \in \widetilde{K}_p$,
\[
((c_1 + c_2) + c_3)(\sigma) = (c_1 + c_2)(\sigma) + c_3(\sigma) = c_1(\sigma) + c_2(\sigma) + c_3(\sigma),
\]
\[
(c_1 + (c_2 + c_3))(\sigma) = c_1(\sigma) + (c_2 + c_3)(\sigma) = c_1(\sigma) + c_2(\sigma) + c_3(\sigma).
\]
Hence, $(c_1 + c_2) + c_3 = c_1 + (c_2 + c_3)$.
\item \textbf{Identity element:} Define $0 \in C_p(K)$ by $0(\sigma) = 0$ for all $\sigma \in \widetilde{K}_p$. Then for any $c \in C_p(K)$,
\[
(c + 0)(\sigma) = c(\sigma) + 0 = c(\sigma).
\]
Hence, $0$ is the identity element.
\item  \textbf{Inverse element:} Let $c \in C_p(K)$. Define $-c \in C_p(K)$ by $(-c)(\sigma) = -c(\sigma)$ for all $\sigma \in \widetilde{K}_p$. Then,
\[
(c + (-c))(\sigma) = c(\sigma) + (-c)(\sigma) = c(\sigma) - c(\sigma) = 0,
\]
so $c + (-c) = 0$, and $-c$ is the inverse of $c$.
\item \textbf{Commutativity:} Let $c_1, c_2 \in C_p(K)$. For all $\sigma \in \widetilde{K}_p$,
\[
(c_1 + c_2)(\sigma) = c_1(\sigma) + c_2(\sigma) = c_2(\sigma) + c_1(\sigma) = (c_2 + c_1)(\sigma).
\]
Thus, $c_1 + c_2 = c_2 + c_1$.
\end{enumerate}
Since all the group axioms are satisfied, $C_p(K)$ is an abelian group under pointwise addition.
\end{proof}

\begin{theorem}\label{thm:CpK-is-free}
Let $K$ be a simplicial complex. Then the $p$-chain group $C_p(K)$ is a free abelian group, and its basis is the set of oriented $p$-simplices in $\widetilde{K}_p$.
\end{theorem}

\begin{proof}
A group is free if it has a basis, i.e., a set of elements such that every element of the group can be uniquely written as a finite linear combination of these basis elements, with integer coefficients.
\begin{enumerate}
\item \textbf{$C_p(K)$ is free.} We can write each element of $C_p(K)$ as a function from the set of oriented $p$-simplices, $\widetilde{K}_p$, to $\mathbb{Z}$. Each function $c \in C_p(K)$ assigns an integer to each oriented $p$-simplex in $K$ and is determined by the values it takes on these oriented simplices.

For each oriented $p$-simplex $\sigma \in \widetilde{K}_p$, consider the function $\delta_\sigma \in C_p(K)$ defined by:
\[
\delta_\sigma(\tau) =
\begin{cases}
1 & \text{if } \tau = \sigma, \\
0 & \text{otherwise}.
\end{cases}
\]
The functions $\delta_\sigma$ for each $\sigma \in \widetilde{K}_p$ form a set of elements in $C_p(K)$. This set is linearly independent. Indeed, if a linear combination of these functions equals the zero function, that is, if:
\[
\sum_{i=1}^n a_i \delta_{\sigma_i} = 0 \quad \text{for some integers } a_1, a_2, \dots, a_n,
\]
then it must be that each $a_i = 0$, because $\delta_{\sigma_i}(\sigma_j) = \delta_{ij}$, which implies that each coefficient $a_i$ must be zero for the sum to be zero at each $\sigma_j$.

Thus, the set $\{ \delta_\sigma : \sigma \in \widetilde{K}_p \}$ is a linearly independent set.
\item \textbf{Basis of $C_p(K)$.} We now show that this set of functions forms a basis for $C_p(K)$. To do this, note that every element of $C_p(K)$ is a function $c : \widetilde{K}_p \to \mathbb{Z}$, and any such function can be written uniquely as a finite linear combination of the functions $\delta_\sigma$:
\[
c = \sum_{\sigma \in \widetilde{K}_p} c(\sigma) \delta_\sigma,
\]
where the sum is finite because $c(\sigma) = 0$ for all but finitely many $\sigma$. Hence, every element of $C_p(K)$ can be uniquely expressed as a linear combination of the $\delta_\sigma$'s.
\end{enumerate}
Thus, $C_p(K)$ is a free abelian group with basis $\{ \delta_\sigma : \sigma \in \widetilde{K}_p \}$, the set of oriented $p$-simplices of $K$.
\end{proof}

The elementary chain serves as the building block of the chain group $C_p(K)$. This makes the structure of $C_p(K)$ more transparent and manageable. It allows us to define a canonical basis for the free abelian group $C_p(K)$. Every $p$-chain can be uniquely written as a finite linear combination of elementary chains, each corresponding to an oriented $p$-simplex. Actually, the elementary chains form a standard basis of $C_p(K)$, just like $\mathbf{e}_i$ vectors form a basis of $\mathbb{R}^n$.

\begin{definition}[Elementary Chain]\label{def:elementary-chain}
Let $K$ be a simplicial complex and $\widetilde{K}_p$ be the set of oriented $p$-simplices of $K$. For an oriented $p$-simplex $\sigma \in \widetilde{K}_p$, the \emph{elementary $p$-chain} corresponding to $\sigma$ is the function $c : \widetilde{K}_p \to \mathbb{Z}$ defined by
\[
c(\tau) =
\begin{cases}
1 & \text{if } \tau = \sigma, \\
-1 & \text{if } \tau \text{ is the opposite orientation of } \sigma, \\
0 & \text{otherwise}.
\end{cases}
\]
We denote this elementary chain also by $\sigma$ when there is no ambiguity.
\end{definition}

\begin{example}[Elementary 1-chains and Basis of $C_1(K)$]
Let $K$ be a simplicial complex formed by a triangle with vertices 
$\mathbf{a}_0, \mathbf{a}_1, \mathbf{a}_2$.

The set of $1$-simplices is:
\[
K_1 = \{ \{ \mathbf{a}_0, \mathbf{a}_1 \},\; \{ \mathbf{a}_1, \mathbf{a}_2 \},\; \{ \mathbf{a}_2, \mathbf{a}_0 \} \}.
\]

The set of oriented $1$-simplices is:
\[
\widetilde{K}_1 = \{ 
[\mathbf{a}_0, \mathbf{a}_1],\;
[\mathbf{a}_1, \mathbf{a}_2],\;
[\mathbf{a}_2, \mathbf{a}_0],\;
[\mathbf{a}_1, \mathbf{a}_0],\;
[\mathbf{a}_2, \mathbf{a}_1],\;
[\mathbf{a}_0, \mathbf{a}_2]
\}.
\]

We fix the orientation:
\[
\sigma_1 = [\mathbf{a}_0, \mathbf{a}_1],\quad
\sigma_2 = [\mathbf{a}_1, \mathbf{a}_2],\quad
\sigma_3 = [\mathbf{a}_2, \mathbf{a}_0].
\]

For each $\sigma_i$, the elementary $1$-chain $\delta_{\sigma_i}$ is defined as:
\[
\delta_{\sigma_i}(\tau) = 
\begin{cases}
1 & \text{if } \tau = \sigma_i, \\
-1 & \text{if } \tau = \bar{\sigma}_i, \\
0 & \text{otherwise}.
\end{cases}
\]

Thus, the elementary chains (basis elements) of $C_1(K)$ are:
\begin{align*}
\delta_{\sigma_1} &= 
\begin{cases}
1 & \text{on } [\mathbf{a}_0, \mathbf{a}_1], \\
-1 & \text{on } [\mathbf{a}_1, \mathbf{a}_0], \\
0 & \text{otherwise},
\end{cases} \\
\delta_{\sigma_2} &= 
\begin{cases}
1 & \text{on } [\mathbf{a}_1, \mathbf{a}_2], \\
-1 & \text{on } [\mathbf{a}_2, \mathbf{a}_1], \\
0 & \text{otherwise},
\end{cases} \\
\delta_{\sigma_3} &= 
\begin{cases}
1 & \text{on } [\mathbf{a}_2, \mathbf{a}_0], \\
-1 & \text{on } [\mathbf{a}_0, \mathbf{a}_2], \\
0 & \text{otherwise}.
\end{cases}
\end{align*}

Therefore, the chain group $C_1(K)$ is the free abelian group with basis:
\[
\left\{ \delta_{[\mathbf{a}_0, \mathbf{a}_1]},\; \delta_{[\mathbf{a}_1, \mathbf{a}_2]},\; \delta_{[\mathbf{a}_2, \mathbf{a}_0]} \right\},
\]
and any $1$-chain $c \in C_1(K)$ can be uniquely expressed as:
\[
c = a_1 \delta_{[\mathbf{a}_0, \mathbf{a}_1]} + a_2 \delta_{[\mathbf{a}_1, \mathbf{a}_2]} + a_3 \delta_{[\mathbf{a}_2, \mathbf{a}_0]}, \quad \text{where } a_1, a_2, a_3 \in \mathbb{Z}.
\]
\end{example}

After introducing the elementary chain for a simplex, we can derive more interesting results related to the computation of a basis for the $p$-chain group. The following result helps us determine the basis of the chain group using elementary chains.

\begin{lemma}[Computing Bases of Chain Groups via Elementary Chains]\label{l:Computing Bases of Chain Groups via Elementary Chains}
A basis for $p$-chain group $C_{p}(K)$ can be obtained by orienting each $p$-simplex and using the corresponding elementary chains as a basis.   
\end{lemma}
\begin{proof}
Let $K$ be a simplicial complex, and let $\widetilde{K}_p$ denote the set of oriented $p$-simplices of $K$. The goal of this proof is to show that the set of elementary chains corresponding to the oriented $p$-simplices forms a basis for the $p$-chain group $C_p(K)$.

\textbf{Step 1: Definition of $p$-chain group $C_p(K)$}

The \textit{$p$-chain group} $C_p(K)$ consists of formal sums of oriented $p$-simplices with integer coefficients. That is, an element $c \in C_p(K)$ can be written as:
\[
c = \sum_{\sigma \in \widetilde{K}_p} a_\sigma \cdot \sigma,
\]
where:
- $\sigma$ is an oriented $p$-simplex from the set $\widetilde{K}_p$,
- $a_\sigma \in \mathbb{Z}$ are integer coefficients.

Thus, $C_p(K)$ is the free abelian group generated by the set of oriented $p$-simplices $\widetilde{K}_p$, and each element of $C_p(K)$ is a formal sum (linear combination) of these oriented simplices.

\textbf{Step 2: Definition of Elementary Chains}

For each oriented $p$-simplex $\sigma \in \widetilde{K}_p$, we define the \textit{elementary chain} $\delta_\sigma$ as a function from the set of all oriented simplices to the integers $\mathbb{Z}$, which is defined as:
\[
\delta_\sigma(\tau) = 
\begin{cases}
1 & \text{if } \tau = \sigma, \\
-1 & \text{if } \tau = \bar{\sigma}, \text{ where } \bar{\sigma} \text{ is the opposite orientation of } \sigma, \\
0 & \text{otherwise}.
\end{cases}
\]
Thus, $\delta_\sigma$ is a function that "picks out" the oriented simplex $\sigma$ and assigns it the coefficient 1, assigns the opposite orientation $\bar{\sigma}$ the coefficient -1, and assigns all other simplices the coefficient 0.

\textbf{Step 3: The Set of Elementary Chains Forms a Free Abelian Group}

The set of elementary chains $\{ \delta_\sigma : \sigma \in \widetilde{K}_p \}$ forms a \textit{free abelian group}. To show this, we must verify the following two conditions:
1. \textit{Linear Independence:}  
   If a linear combination of elementary chains is equal to the zero chain, then the coefficients must all be zero. More formally, if:
   \[
   \sum_{i} a_i \delta_{\sigma_i} = 0,
   \]
   then $a_i = 0$ for all $i$. This is because the elementary chains are defined to be nonzero only on the specific simplices they correspond to, so for the sum to be zero, each coefficient must vanish.
   
2. \textit{Free Group Structure:}  
   There is a unique way to express any $p$-chain as a linear combination of elementary chains. Given that the elementary chains are linearly independent, they generate a free abelian group.

Thus, the set $\{ \delta_\sigma : \sigma \in \widetilde{K}_p \}$ is a \textit{basis} for the free abelian group $C_p(K)$.

\textbf{Step 4: Orientation of $p$-simplices}

Each $p$-simplex $\sigma \in K_p$ is \textit{oriented} in the sense that it is given a direction, which leads to an ordered set of vertices. The orientation distinguishes $\sigma$ from its opposite $\bar{\sigma}$. Since we consider the elementary chain associated with each \textit{oriented} $p$-simplex, the basis elements $\delta_\sigma$ correspond to these oriented simplices.

The fact that the simplices are oriented ensures that each simplex has a distinct representation in the chain group, as $\sigma$ and $\bar{\sigma}$ are treated as different elements in the basis.

\textbf{Step 5: Linear Independence and Spanning}

The set $\{ \delta_\sigma : \sigma \in \widetilde{K}_p \}$ is linearly independent, as shown in Step 3. Furthermore, any $p$-chain $c \in C_p(K)$ can be expressed as a unique linear combination of these elementary chains. Specifically, any $p$-chain $c$ can be written as:
\[
c = \sum_{\sigma \in \widetilde{K}_p} a_\sigma \cdot \delta_\sigma,
\]
where the coefficients $a_\sigma \in \mathbb{Z}$ correspond to the integer coefficients of the simplices in the chain.

Thus, the set $\{ \delta_\sigma : \sigma \in \widetilde{K}_p \}$ \textit{spans} the $p$-chain group $C_p(K)$.

\textbf{Step 6: Conclusion}

Since the set $\{ \delta_\sigma : \sigma \in \widetilde{K}_p \}$ is both linearly independent and spans $C_p(K)$, it forms a basis for the $p$-chain group. Therefore, we have shown that:

A basis for  $C_p(K)$ can be obtained by orienting each $p$-simplex and using the corresponding elementary chains as a basis.
\end{proof}

\begin{remark} \hfill
\begin{enumerate}
\item In $C_0(K)$, the basis elements are vertex-based, while in $C_p(K)$ for $p > 0$, the basis elements are simplicial-based (corresponding to the oriented $p$-simplices).

\item The basis for $C_0(K)$ is more straightforward and deals with the zero-dimensional components (vertices), whereas the basis for $C_p(K)$ captures higher-dimensional geometric objects (simplices) and their orientations.
\end{enumerate}
\end{remark}

The next result \eqref{co:Universal Mapping Property of Chain Group} shows that the chain group $C_p(K)$ has a universal mapping property: Any function from the basis elements (the oriented $p$-simplices) to an abelian group $G$ that respects orientation can be extended uniquely to a group homomorphism from the whole chain group $C_p(K)$ to $G$.
\begin{corollary}[Universal Mapping Property of Chain Group]\label{co:Universal Mapping Property of Chain Group}
Any function $f$ from the oriented $p$-simplices of $K$ to an abelian group $G$ extends uniquely to a homomorphism $\phi \colon C_{p}(K) \to G$, provided that $f(-\sigma) = -f(\sigma)$ for all oriented $p$-simplices $\sigma$. 
\end{corollary}

If we want to define a map $\phi \colon C_p(K) \to \mathbb{Z}$ (e.g., to count something, assign weights, or compute evaluations), this Corollary  \eqref{co:Universal Mapping Property of Chain Group} tells us: Just define $f(\sigma)$ for each oriented simplex $\sigma$ so that $f(-\sigma) = -f(\sigma)$, and we automatically get a well-defined and unique homomorphism $\phi$. For better understanding let us look at an example that describes the universal property of chain groups. 

\begin{example}\label{eq:Universal Mapping Property of Chain Group}
Let the simplicial complex $K$ consist of three vertices and three edges forming a triangle:
\[K = \{\{\mathbf{a}_0\}, \{\mathbf{a}_1\}, \{\mathbf{a}_2\}, \{\mathbf{a}_0, \mathbf{a}_1\}, \{\mathbf{a}_1, \mathbf{a}_2\}, \{\mathbf{a}_2, \mathbf{a}_0\}\}\]
Choose oriented edges:
\[
\sigma_1 = [\mathbf{a}_0, \mathbf{a}_1], \quad \sigma_2 = [\mathbf{a}_1, \mathbf{a}_2], \quad \sigma_3 = [\mathbf{a}_2, \mathbf{a}_0]
\]
Then the set of oriented 1-simplices is $\widetilde{K}_1 = \{ \pm \sigma_1, \pm \sigma_2, \pm \sigma_3 \}$.

The 1-chain group $C_1(K)$ is the free abelian group generated by $\{ \sigma_1, \sigma_2, \sigma_3 \}$, with the relation that $-\sigma = \sigma'$. Define a function $f \colon \widetilde{K}_1 \to \mathbb{Z}$ as follows:
\begin{gather*}
f(\sigma_1) = 2, f(\sigma_2) = -1,
f(\sigma_3) = 3, 
\text{and } f(-\sigma) = -f(\sigma) \text{ for all } \sigma
\end{gather*}
Thus,
\[
f(-\sigma_1) = -2, f(-\sigma_2) = 1,
f(-\sigma_3) = -3
\]
By the Corollary  \eqref{co:Universal Mapping Property of Chain Group}, the function $f$ extends uniquely to a homomorphism $\phi: C_1(K) \to \mathbb{Z}$ defined by linearity. Let $c = n_1 \sigma_1 + n_2 \sigma_2 + n_3 \sigma_3 \in C_1(K)$ be a 1-chain. Then
\[
\phi(c) = n_1 f(\sigma_1) + n_2 f(\sigma_2) + n_3 f(\sigma_3)
\]
That is,
\[
\phi(c) = n_1 \cdot 2 + n_2 \cdot (-1) + n_3 \cdot 3
\]
Particularly, let us consider  $c = 3\sigma_1 - 2\sigma_2 + \sigma_3$, then
\[
\phi(c) = 3 \cdot 2 + (-2) \cdot (-1) + 1 \cdot 3 = 6 + 2 + 3 = 11
\]
This example demonstrates how:
\begin{itemize}
    \item A function $f \colon \widetilde{K}_1 \to \mathbb{Z}$ on oriented $p$-simplices that satisfies $f(-\sigma) = -f(\sigma)$ extends uniquely to a group homomorphism $\phi \colon C_p(K) \to \mathbb{Z}$.
    \item The extension is linear and determined by the values of $f$ on the generators of $C_p(K)$.
\end{itemize}
\end{example}
\section{Boundary Operator}\label{s:Boundary Operator}
The objective of boundary operators in the context of simplicial complexes and chain groups is to algebraically capture how simplices are connected to one another by mapping a $p$-simplex to its boundary, which is a formal sum of its $(p-1)$-dimensional faces. It helps us move from a $p$-chain to a $(p-1)$-chain, reflecting how higher-dimensional simplices are ``built from" lower-dimensional ones. It is a fundamental tool to study topological features like holes and connectedness, which are formalized later through homology groups. 

\begin{definition}[Boundary Operator]
Let $K$ be an oriented simplicial complex. For each $p \geq 1$, the \emph{boundary operator}
\[
\partial_p : C_p(K) \to C_{p-1}(K)
\]
is the group homomorphism defined on an elementary $p$-simplex $\sigma = [\mathbf{a}_0, \mathbf{a}_1, \dots, \mathbf{a}_p]$ by
\[
\partial_p(\sigma) = \sum_{i=0}^{p} (-1)^i [\mathbf{a}_0, \dots, \widehat{\mathbf{a}_i}, \dots, \mathbf{a}_p],
\]
where $\widehat{\mathbf{a}_i}$ indicates that the vertex $\mathbf{a}_i$ is omitted from the sequence. That is, each term in the sum is a $(p-1)$-simplex obtained by deleting exactly one vertex from $\sigma$, with a sign determined by its position.
\end{definition}
\begin{remark}\hfill
\begin{enumerate}
\item This definition extends linearly to all $p$-chains in $C_p(K)$, i.e., for $c = \sum n_j \sigma_j \in C_p(K)$, we define
\[
\partial_p(c) = \sum n_j \, \partial_p(\sigma_j).
\]    
\item Since $C_{p}(K)$ is trivial group for $p < 0$, the operator $\partial_p$ is the trivial homomorphsim for $p \leq 0$. Let us see the complete explanation about this fact.

The chain group $C_p(K)$ is defined for each non-negative integer $p$ as the free abelian group generated by the oriented $p$-simplices of a simplicial complex $K$. This means that for $p = 0$, $C_0(K)$ consists of formal integer linear combinations of the vertices of $K$, and for $p = 1$, $C_1(K)$ consists of formal combinations of edges, and so on. However, when we consider values of $p < 0$, there are no simplices of negative dimension, so we define $C_p(K)$ to be the trivial group (i.e., the group containing only the zero element) for all $p < 0$.

Now, in the context of the boundary operator $\partial_p : C_p(K) \to C_{p-1}(K)$, this operator must be defined for all $p \in \mathbb{Z}$ if we want to describe the full chain complex in algebraic topology. But if $p \leq 0$, then either the domain $C_p(K)$ or the codomain $C_{p-1}(K)$ (or both) is trivial.

\begin{enumerate}
    \item When $p = 0$, the boundary map $\partial_0 : C_0(K) \to C_{-1}(K)$ maps into a trivial group, because $C_{-1}(K) = {0}$. Hence, the map $\partial_0$ sends every element of $C_0(K)$ to $0$, making it the zero homomorphism.
    \item When $p < 0$, both $C_p(K)$ and $C_{p-1}(K)$ are trivial groups. The only possible homomorphism between trivial groups is the trivial (zero) homomorphism.
\end{enumerate}
Therefore, in order to maintain consistency and define the chain complex for all integers $p$, we adopt the convention that $C_p(K)$ is the trivial group for $p < 0$, and hence $\partial_p$ is also the trivial homomorphism in those cases. This simplifies many theoretical constructions and ensures that no unexpected exceptions occur when extending results to all integers.
\end{enumerate}

\end{remark}

\begin{question}
Is it possible to define the boundary operator for $p = 0$. 
\end{question}

\begin{answer}
Yes, it is indeed possible to take $p = 0$ when defining the boundary operator in the context of simplicial homology. In fact, doing so is essential for the completeness of the chain complex. Recall that the boundary operator $\partial_p$ is defined as a homomorphism from the $p$-chain group $C_p(K)$ to the $(p-1)$-chain group $C_{p-1}(K)$, capturing the idea of a boundary of a $p$-simplex as a sum of its $(p-1)$-faces with appropriate orientation. However, when $p = 0$, the objects in $C_0(K)$ are formal linear combinations of $0$-simplices, i.e., the vertices of the simplicial complex. Since a $0$-simplex (a point) has no faces, there is no natural way to define its boundary in terms of lower-dimensional simplices. To handle this situation, we define the boundary operator $\partial_0 \colon C_0(K) \to 0$ as the zero map. That is, for any $0$-chain $c \in C_0(K)$, we set $\partial_0(c) = 0$. We will look at this further in the discussion of chain complexes for more details.
\end{answer}

\begin{theorem}[Boundary Operator as Homomorphism]\label{t:Boundary Operator as Homomorphism}
The boundary operator $\partial_p : C_p(K) \to C_{p-1}(K)$ is a well-defined group homomorphism for each integer $p \geq 1$. Moreover, for any oriented $p$-simplex $\sigma$, the boundary satisfies $\partial_p(-\sigma) = -\partial_p(\sigma)$.
\end{theorem}

\begin{proof}
Let $\sigma = [\mathbf{a}_0, \dots, \mathbf{a}_p]$ be an oriented $p$-simplex in $K$. The boundary operator is defined on $\sigma$ by
\[
\partial_p(\sigma) = \sum_{i=0}^{p} (-1)^i [\mathbf{a}_0, \dots, \widehat{\mathbf{a}_i}, \dots, \mathbf{a}_p],
\]
where $\widehat{\mathbf{a}_i}$ denotes the omission of the $i$-th vertex. Each face is given the induced orientation from $\sigma$.

To verify that $\partial_p$ is well-defined, we must ensure that it respects the orientation of simplices. Suppose $-\sigma$ denotes the same simplex as $\sigma$ with opposite orientation. Then we define
\[
\partial_p(-\sigma) := -\partial_p(\sigma).
\]
This ensures that the operator is consistent with the antisymmetric nature of orientation and thus well-defined on the free abelian group generated by oriented $p$-simplices. Hence,
\[
\partial_p(-\sigma) = - \sum_{i=0}^p (-1)^i [\mathbf{a}_0, \dots, \widehat{\mathbf{a}_i}, \dots, \mathbf{a}_p] = -\partial_p(\sigma).
\]

Now, let $c = \sum_{i=1}^n n_i \sigma_i$ and $c' = \sum_{j=1}^m m_j \tau_j$ be two arbitrary $p$-chains in $C_p(K)$, where $n_i, m_j \in \mathbb{Z}$ and each $\sigma_i$, $\tau_j$ is an oriented $p$-simplex in $K$. Then,
\[
c + c' = \sum_{i=1}^n n_i \sigma_i + \sum_{j=1}^m m_j \tau_j.
\]
Applying the boundary operator linearly:
\[
\partial_p(c + c') = \partial_p\left( \sum_{i=1}^n n_i \sigma_i + \sum_{j=1}^m m_j \tau_j \right)
= \sum_{i=1}^n n_i \partial_p(\sigma_i) + \sum_{j=1}^m m_j \partial_p(\tau_j)
= \partial_p(c) + \partial_p(c').
\]
Thus, $\partial_p$ preserves the group operation (addition) in $C_p(K)$. Hence, $\partial_p$ is a well-defined group homomorphism, and it satisfies $\partial_p(-\sigma) = -\partial_p(\sigma)$.
\end{proof}

\begin{example}[Computation of Boundary Operator for 1-simplex]\label{eq:Computation of Boundary Operator for 1 simplex}
Let $\sigma = [\mathbf{a}_0, \mathbf{a}_1]$ be an oriented $1$-simplex in a simplicial complex $K$. Compute the boundary $\partial_1(\sigma)$.

The $1$-simplex $\sigma = [\mathbf{a}_0, \mathbf{a}_1]$ represents an oriented edge from vertex $\mathbf{a}_0$ to vertex $\mathbf{a}_1$. According to the definition of the boundary operator for $p = 1$ i.e. $\partial_1 \colon C_{1}(K) \to C_{0}(K)$:
\[
\partial_1([\mathbf{a}_0, \mathbf{a}_1]) = (-1)^0 [\mathbf{a}_1] + (-1)^1 [\mathbf{a}_0] = [\mathbf{a}_1] - [\mathbf{a}_0].
\]
Thus, the boundary of the $1$-simplex is the formal difference of its two endpoints:
\[
\partial_1([\mathbf{a}_0, \mathbf{a}_1]) = [\mathbf{a}_1] - [\mathbf{a}_0].
\]
This expression lies in the chain group $C_0(K)$, which is generated by the $0$-simplices (i.e., the vertices of $K$). It reflects that the ``boundary of an oriented edge is the difference between its terminal vertex and its initial vertex".

Note that changing the orientation of the $1$-simplex reverses the sign:
\[
\partial_1([\mathbf{a}_1, \mathbf{a}_0]) = [\mathbf{a}_0] - [\mathbf{a}_1] = -\partial_1([\mathbf{a}_0, \mathbf{a}_1]).
\]
This illustrates the anti-symmetry property of the boundary operator with respect to orientation, i.e., $\partial_p(-\sigma) = -\partial_p(\sigma)$.    
\end{example}

\begin{example}[Computation of Boundary Operator for 2-simplex]\label{eq:Computation of Boundary Operator for 2 simplex}
Let $\sigma = [\mathbf{a}_0, \mathbf{a}_1, \mathbf{a}_2]$ be an oriented $2$-simplex in a simplicial complex $K$. Compute the boundary $\partial_2(\sigma)$.

We are given a $2$-simplex $\sigma = [\mathbf{a}_0, \mathbf{a}_1, \mathbf{a}_2]$, which geometrically represents an oriented triangle with vertices $\mathbf{a}_0$, $\mathbf{a}_1$, and $\mathbf{a}_2$, in that cyclic order. By the definition of the boundary operator for $p = 2$ i.e. $\partial_2 \colon C_{2}(K) \to C_{1}(K)$:
\[
\partial_2([\mathbf{a}_0, \mathbf{a}_1, \mathbf{a}_2]) = 
(-1)^0 [\mathbf{a}_1, \mathbf{a}_2] 
+ (-1)^1 [\mathbf{a}_0, \mathbf{a}_2] 
+ (-1)^2 [\mathbf{a}_0, \mathbf{a}_1].
\]
That is,
\[
\partial_2([\mathbf{a}_0, \mathbf{a}_1, \mathbf{a}_2]) = 
[\mathbf{a}_1, \mathbf{a}_2] 
- [\mathbf{a}_0, \mathbf{a}_2] 
+ [\mathbf{a}_0, \mathbf{a}_1].
\]
This result lies in the chain group $C_1(K)$, which is generated by the oriented $1$-simplices (edges). Each term in the sum is a face of the triangle, with the orientation ``induced" from the ordering of the triangle’s vertices. Where 
\begin{itemize}
\item $[\mathbf{a}_1, \mathbf{a}_2]$ is the oriented edge opposite to vertex $\mathbf{a}_0$ (obtained by omitting $\mathbf{a}_0$),
\item $-[\mathbf{a}_0, \mathbf{a}_2]$ is the oriented edge opposite to vertex $\mathbf{a}_1$ (obtained by omitting $\mathbf{a}_1$),
\item $[\mathbf{a}_0, \mathbf{a}_1]$ is the oriented edge opposite to vertex $\mathbf{a}_2$ (obtained by omitting $\mathbf{a}_2$).
\end{itemize}
These orientations are consistent with the induced orientations from the $2$-simplex. So, the boundary of the oriented triangle is the formal sum:
\[
\partial_2([\mathbf{a}_0, \mathbf{a}_1, \mathbf{a}_2]) = 
[\mathbf{a}_1, \mathbf{a}_2] 
- [\mathbf{a}_0, \mathbf{a}_2] 
+ [\mathbf{a}_0, \mathbf{a}_1],
\]
which algebraically encodes the boundary path around the triangle, traversing its edges in an oriented fashion.

If we were to reverse the orientation of the 2-simplex, we would get:
\[
\partial_2([\mathbf{a}_0, \mathbf{a}_2, \mathbf{a}_1]) = 
[\mathbf{a}_2, \mathbf{a}_1] 
- [\mathbf{a}_0, \mathbf{a}_1] 
+ [\mathbf{a}_0, \mathbf{a}_2] = -\partial_2([\mathbf{a}_0, \mathbf{a}_1, \mathbf{a}_2]).
\]
Thus, the boundary operator respects orientation, and $\partial_2(-\sigma) = -\partial_2(\sigma)$.
\end{example}

\begin{example}[Computation of Boundary Operator for 3-simplex]\label{eq:Computation of Boundary Operator for 3 simplex}
Let $\sigma = [\mathbf{a}_0, \mathbf{a}_1, \mathbf{a}_2, \mathbf{a}_3]$ be a $3$-simplex with the ordered vertices $\mathbf{a}_0, \mathbf{a}_1, \mathbf{a}_2, \mathbf{a}_3$. We compute the boundary $\partial_3(\sigma)$ using the definition of the boundary operator.

The boundary operator $\partial_p \colon C_{p}(K) \to C_{p- 1}(K)$ for a $p$-simplex $[\mathbf{a}_0, \dots, \mathbf{a}_p]$ is given by:
\[
\partial_p([\mathbf{a}_0, \dots, \mathbf{a}_p]) = \sum_{i=0}^{p} (-1)^i [\mathbf{a}_0, \dots, \widehat{\mathbf{a}_i}, \dots, \mathbf{a}_p],
\]
where $\widehat{\mathbf{a}_i}$ means that the vertex $\mathbf{a}_i$ is omitted.

For $p = 3$ the boundary operator $\partial_3 \colon C_{3}(K) \to C_{2}(K)$, we have:
\begin{align*}
\partial_3([\mathbf{a}_0, \mathbf{a}_1, \mathbf{a}_2, \mathbf{a}_3]) &= 
(-1)^0 [\mathbf{a}_1, \mathbf{a}_2, \mathbf{a}_3] +
(-1)^1 [\mathbf{a}_0, \mathbf{a}_2, \mathbf{a}_3] \\
&\quad + (-1)^2 [\mathbf{a}_0, \mathbf{a}_1, \mathbf{a}_3] +
(-1)^3 [\mathbf{a}_0, \mathbf{a}_1, \mathbf{a}_2] \\
&= [\mathbf{a}_1, \mathbf{a}_2, \mathbf{a}_3]
- [\mathbf{a}_0, \mathbf{a}_2, \mathbf{a}_3]
+ [\mathbf{a}_0, \mathbf{a}_1, \mathbf{a}_3]
- [\mathbf{a}_0, \mathbf{a}_1, \mathbf{a}_2]
\end{align*}

Each term in the sum corresponds to one of the triangular faces of the tetrahedron:
\begin{itemize}
    \item $[\mathbf{a}_1, \mathbf{a}_2, \mathbf{a}_3]$ is the face opposite $\mathbf{a}_0$,
    \item $-[\mathbf{a}_0, \mathbf{a}_2, \mathbf{a}_3]$ is the face opposite $\mathbf{a}_1$,
    \item $[\mathbf{a}_0, \mathbf{a}_1, \mathbf{a}_3]$ is the face opposite $\mathbf{a}_2$,
    \item $-[\mathbf{a}_0, \mathbf{a}_1, \mathbf{a}_2]$ is the face opposite $\mathbf{a}_3$.
\end{itemize}
The boundary operator algebraically encodes how each higher-dimensional simplex can be decomposed into lower-dimensional faces. For example, the boundary of the $3$-simplex $\sigma = [\mathbf{a}_0, \mathbf{a}_1, \mathbf{a}_2, \mathbf{a}_3]$ is the sum of the oriented $2$-faces (triangles) that form its boundary, with appropriate signs reflecting their orientation relative to the original simplex. This operation reflects the underlying structure of the simplicial complex, where each simplex is part of a boundary of a higher-dimensional simplex.

If we were to reverse the orientation of the $3$-simplex $\sigma = [\mathbf{a}_0, \mathbf{a}_1, \mathbf{a}_2, \mathbf{a}_3]$, we would multiply the simplex by $-1$, and the new oriented simplex would be $-[\mathbf{a}_0, \mathbf{a}_1, \mathbf{a}_2, \mathbf{a}_3]$. For the boundary operator, we would have:
\[
\partial_3([\mathbf{a}_0, \mathbf{a}_1, \mathbf{a}_2, \mathbf{a}_3]) = [\mathbf{a}_1, \mathbf{a}_2, \mathbf{a}_3] - [\mathbf{a}_0, \mathbf{a}_2, \mathbf{a}_3] + [\mathbf{a}_0, \mathbf{a}_1, \mathbf{a}_3] - [\mathbf{a}_0, \mathbf{a}_1, \mathbf{a}_2],
\]
and for the reversed orientation:
\[
\partial_3(-[\mathbf{a}_0, \mathbf{a}_1, \mathbf{a}_2, \mathbf{a}_3]) = -[\mathbf{a}_1, \mathbf{a}_2, \mathbf{a}_3] + [\mathbf{a}_0, \mathbf{a}_2, \mathbf{a}_3] - [\mathbf{a}_0, \mathbf{a}_1, \mathbf{a}_3] + [\mathbf{a}_0, \mathbf{a}_1, \mathbf{a}_2].
\]
Thus, reversing the orientation of the $3$-simplex changes the signs of the boundary operator terms.

Hence, the boundary of the $3$-simplex is a $2$-chain consisting of all the oriented faces with appropriate signs determined by the orientation.
\end{example}

The objective of this lemma is to show that if we apply two successive boundary operators, one after the other, on a simplex or a chain, the result is always zero. In simpler terms, the boundary of a boundary vanishes. This is fundamental to the construction of chain complexes, where it ensures that the boundary maps are consistent and well-defined.

\begin{lemma}[Boundary of a Boundary is Zero]\label{l:Boundary of a Boundary is Zero}
Let $K$ be a simplicial complex and $\partial_p \colon C_p(K) \to C_{p-1}(K)$ be the boundary operator on $p$-chains. Then, for all integers $p \geq 1$, we have
\[
\partial_{p-1} \circ \partial_p = 0.
\]
That is, for any $p$-chain $c \in C_p(K)$,
\[
\partial_{p-1}(\partial_p(c)) = 0.
\]
\end{lemma}

\begin{proof}
Since the boundary operator is linear, it suffices to prove the result for an elementary $p$-chain, that is, for a single oriented $p$-simplex 
\[
\sigma = [\mathbf{a}_0, \mathbf{a}_1, \dots, \mathbf{a}_p].
\]
By the definition of the boundary operator, we have:
\[
\partial_p(\sigma) = \sum_{i=0}^{p} (-1)^i [\mathbf{a}_0, \dots, \widehat{\mathbf{a}_i}, \dots, \mathbf{a}_p],
\]
where $\widehat{\mathbf{a}_i}$ denotes that the vertex $\mathbf{a}_i$ is omitted.

Now, we apply the boundary operator $\partial_{p-1}$ to each $(p-1)$-simplex in the sum:
\[
\partial_{p-1}(\partial_p(\sigma)) = \sum_{i=0}^{p} (-1)^i \partial_{p-1}([\mathbf{a}_0, \dots, \widehat{\mathbf{a}_i}, \dots, \mathbf{a}_p]).
\]
Now, fix $i$ and consider the term 
\[
\partial_{p-1}([\mathbf{a}_0, \dots, \widehat{\mathbf{a}_i}, \dots, \mathbf{a}_p]) = \sum_{j=0}^{p-1} (-1)^j [\mathbf{a}_0, \dots, \widehat{\mathbf{a}_i}, \dots, \widehat{\mathbf{a}_j}, \dots, \mathbf{a}_p].
\]
Each term in this double sum corresponds to a $(p-2)$-simplex obtained by omitting two distinct vertices from $\sigma$. We now show that each such $(p-2)$-simplex appears twice in the total sum, but with opposite signs.

Let $\tau = [\mathbf{a}_0, \dots, \widehat{\mathbf{a}_i}, \dots, \widehat{\mathbf{a}_j}, \dots, \mathbf{a}_p]$ with $i < j$. Then, $\tau$ appears:
\begin{itemize}
\item Once in the expansion of $\partial_{p-1}$ applied to the term where $\mathbf{a}_i$ was omitted first, and then $\mathbf{a}_j$ (adjusted for indexing after $\mathbf{a}_i$ is removed).
\item Once in the expansion where $\mathbf{a}_j$ is omitted first, and then $\mathbf{a}_i$.
\end{itemize}
These two terms have opposite signs due to the alternating signs in the boundary formula:
\[
(-1)^i (-1)^{j-1} + (-1)^j (-1)^i = (-1)^{i + j - 1} + (-1)^{i + j} = 0.
\]
Hence, every $(p-2)$-simplex appears twice with opposite signs, and so the total sum is zero. Thus,
\[
\partial_{p-1}(\partial_p(\sigma)) = 0.
\]
By linearity, it follows that for any $c \in C_p(K)$,
\[
\partial_{p-1}(\partial_p(c)) = 0.
\]
\end{proof}

\begin{example}[Verification of $\partial_2 \circ \partial_3 = 0$ for a 3-simplex]\label{eq:Verification of Lemma for successive boundary operators}
Let $\sigma = [\mathbf{a}_0, \mathbf{a}_1, \mathbf{a}_2, \mathbf{a}_3]$ be an oriented 3-simplex. We first compute the boundary:
\[
\partial_3(\sigma) = 
[\mathbf{a}_1, \mathbf{a}_2, \mathbf{a}_3] 
- [\mathbf{a}_0, \mathbf{a}_2, \mathbf{a}_3] 
+ [\mathbf{a}_0, \mathbf{a}_1, \mathbf{a}_3] 
- [\mathbf{a}_0, \mathbf{a}_1, \mathbf{a}_2].
\]
Now compute $\partial_2(\partial_3(\sigma))$:
\begin{align*}
\partial_2(\partial_3(\sigma)) 
&= \partial_2([\mathbf{a}_1, \mathbf{a}_2, \mathbf{a}_3]) 
- \partial_2([\mathbf{a}_0, \mathbf{a}_2, \mathbf{a}_3]) \\
&\quad + \partial_2([\mathbf{a}_0, \mathbf{a}_1, \mathbf{a}_3]) 
- \partial_2([\mathbf{a}_0, \mathbf{a}_1, \mathbf{a}_2]) \\
&= \big( [\mathbf{a}_2, \mathbf{a}_3] - [\mathbf{a}_1, \mathbf{a}_3] + [\mathbf{a}_1, \mathbf{a}_2] \big) \\
&\quad - \big( [\mathbf{a}_2, \mathbf{a}_3] - [\mathbf{a}_0, \mathbf{a}_3] + [\mathbf{a}_0, \mathbf{a}_2] \big) \\
&\quad + \big( [\mathbf{a}_1, \mathbf{a}_3] - [\mathbf{a}_0, \mathbf{a}_3] + [\mathbf{a}_0, \mathbf{a}_1] \big) \\
&\quad - \big( [\mathbf{a}_1, \mathbf{a}_2] - [\mathbf{a}_0, \mathbf{a}_2] + [\mathbf{a}_0, \mathbf{a}_1] \big).
\end{align*}
Now combine and cancel like terms:
\begin{align*}
[\mathbf{a}_2, \mathbf{a}_3] - [\mathbf{a}_2, \mathbf{a}_3] &= 0, \\
-[\mathbf{a}_1, \mathbf{a}_3] + [\mathbf{a}_1, \mathbf{a}_3] &= 0, \\
[\mathbf{a}_1, \mathbf{a}_2] - [\mathbf{a}_1, \mathbf{a}_2] &= 0, \\
[\mathbf{a}_0, \mathbf{a}_3] - [\mathbf{a}_0, \mathbf{a}_3] &= 0, \\
[\mathbf{a}_0, \mathbf{a}_2] - [\mathbf{a}_0, \mathbf{a}_2] &= 0, \\
[\mathbf{a}_0, \mathbf{a}_1] - [\mathbf{a}_0, \mathbf{a}_1] &= 0.
\end{align*}
Hence, $\partial_2(\partial_3(\sigma)) = 0$.
\end{example}

\section{Homology Groups}\label{s:Homology Groups}

\begin{definition}[$p$-cycles]\label{d:p cycles}
Let $K$ be a simplicial complex and let $C_p(K)$ denote the group of $p$-chains, which are formal sums of oriented $p$-simplices with integer coefficients (or coefficients from any abelian group). A $p$-cycle is a $p$-chain whose boundary is zero. That is, a $p$-chain $c \in C_p(K)$ is called a $p$-cycle if:
\[
\partial_p(c) = 0
\]
Here, $\partial_p \colon C_p(K) \to C_{p-1}(K)$ is the boundary operator, which maps a $p$-chain to its boundary, a $(p-1)$-chain.
\end{definition}

\begin{observation}[Geometric Intuition of $p$-cycles]
Geometrically, a $p$-cycle can be thought of as a closed $p$-dimensional structure made from $p$-simplices. For example:
\begin{enumerate}
\item A 1-cycle is a collection of edges (1-simplices) forming closed loops, like the boundary of a triangle or more complex loops.
\item A 2-cycle is a collection of 2-simplices (triangles) that together form a closed surface, like the surface of a tetrahedron (without its interior).
\end{enumerate}
In simple terms, $p$-cycles are ``loop-like" $p$-dimensional objects with no boundary.
\end{observation}

\begin{definition}[Cycle Group of $p$-cycles]\label{d:Cycle Group of p cycles}
The set of all $p$-cycles in simplicial complex $K$ forms a subgroup of chain group  $C_p(K)$, denoted by:
\[
Z_p(K) = \ker(\partial_p) = \{ c \in C_p(K) \mid \partial_p(c) = 0 \}
\]
This group is called the cycle group of dimension $p$.
\end{definition}

\begin{remark}
By the standard definition of kernel in Group theory, the cycle group $Z_p(K)$ of $p$-cycles is actually the kernel of the boundary operator $\partial \colon C_{p}(K) \to C_{p - 1}(K)$. 
\end{remark}

\begin{example}[1-Cycle]\label{eq:1 Cycle}
Let $K$ be a simplicial complex (tetrahedron) with vertex (oriented ) set $V = \{\btext{a}_0, \btext{a}_1, \btext{a}_2\}$. The 1-simplices (oriented edges) of the traingle are:
\[
\sigma_1 = [\btext{a}_0, \btext{a}_1], \sigma_2 = [\btext{a}_1, \btext{a}_2], \sigma_3 = [\btext{a}_2, \btext{a}_0]
\]
Then the chain $c = \sigma_1 + \sigma_2 + \sigma_3$ is a 1-cycle because:
\begin{align*}
\partial_1(c) & = \partial_1(\sigma_1) + \partial_1(\sigma_2) + \partial_1(\sigma_3) \\
& = [\btext{a}_0] - [\btext{a}_1] + [\btext{a}_1] - [\btext{a}_2] + [\btext{a}_2] - [\btext{a}_0]  \\
& = 0
\end{align*}
Since $\partial_1(c) = 0$, the chain
\[
z = [\btext{a}_0, \btext{a}_1] + [\btext{a}_1, \btext{a}_2] + [\btext{a}_2, \btext{a}_0] 
\]
is a 1-cycle. That is, $c \in Z_2(K)$.
\end{example}

\begin{example}[2-Cycle]\label{eq:2 Cycle}
Let $K$ be a simplicial complex (tetrahedron) with vertex (oriented ) set $V = \{\btext{a}_0, \btext{a}_1, \btext{a}_2, \btext{a}_3\}$. The 2-simplices (oriented triangular faces) of the tetrahedron are:
\[\sigma_1 = [\btext{a}_0, \btext{a}_1, \btext{a}_2],
\sigma_2 = [\btext{a}_0, \btext{a}_1, \btext{a}_3],
\sigma_3 = [\btext{a}_0, \btext{a}_2, \btext{a}_3],
\sigma_4 = [\btext{a}_1, \btext{a}_2, \btext{a}_3]\]
Let us define a 2-chain:
\[
c = \sigma_1 - \sigma_2 + \sigma_3 - \sigma_4
\]
Recall that the boundary of a 2-simplex $[\btext{a}_i, \btext{a}_j, \btext{a}_k]$ is given by:
\[
\partial_2([\btext{a}_i, \btext{a}_j, \btext{a}_k]) = [\btext{a}_j, \btext{a}_k] - [\btext{a}_i, \btext{a}_k] + [\btext{a}_i, \btext{a}_j]
\]
We compute the boundary of $z$:
\begin{align*}
\partial_2(c) &= \partial_2(\sigma_1) - \partial_2(\sigma_2) + \partial_2(\sigma_3) - \partial_2(\sigma_4) \\
&= \left( [\btext{a}_1, \btext{a}_2] - [\btext{a}_0, \btext{a}_2] + [\btext{a}_0, \btext{a}_1] \right) \\
&\quad - \left( [\btext{a}_1, \btext{a}_3] - [\btext{a}_0, \btext{a}_3] + [\btext{a}_0, \btext{a}_1] \right) \\
&\quad + \left( [\btext{a}_2, \btext{a}_3] - [\btext{a}_0, \btext{a}_3] + [\btext{a}_0, \btext{a}_2] \right) \\
&\quad - \left( [\btext{a}_2, \btext{a}_3] - [\btext{a}_1, \btext{a}_3] + [\btext{a}_1, \btext{a}_2] \right)
\end{align*}
Now combine terms:
\begin{align*}
\partial_2(c) &= [\btext{a}_1, \btext{a}_2] - [\btext{a}_0, \btext{a}_2] + [\btext{a}_0, \btext{a}_1] \\
&\quad - [\btext{a}_1, \btext{a}_3] + [\btext{a}_0, \btext{a}_3] - [\btext{a}_0, \btext{a}_1] \\
&\quad + [\btext{a}_2, \btext{a}_3] - [\btext{a}_0, \btext{a}_3] + [\btext{a}_0, \btext{a}_2] \\
&\quad - [\btext{a}_2, \btext{a}_3] + [\btext{a}_1, \btext{a}_3] - [\btext{a}_1, \btext{a}_2]
\end{align*}
Simplify:
\begin{align*}
\partial_2(c) &= \left( [\btext{a}_1, \btext{a}_2] - [\btext{a}_1, \btext{a}_2] \right)
+ \left( -[\btext{a}_0, \btext{a}_2] + [\btext{a}_0, \btext{a}_2] \right)
+ \left( [\btext{a}_0, \btext{a}_1] - [\btext{a}_0, \btext{a}_1] \right) \\
&\quad + \left( -[\btext{a}_1, \btext{a}_3] + [\btext{a}_1, \btext{a}_3] \right)
+ \left( [\btext{a}_0, \btext{a}_3] - [\btext{a}_0, \btext{a}_3] \right)
+ \left( [\btext{a}_2, \btext{a}_3] - [\btext{a}_2, \btext{a}_3] \right) \\
&= 0
\end{align*}
Since $\partial_2(c) = 0$, the chain
\[
z = [\btext{a}_0, \btext{a}_1, \btext{a}_2] - [\btext{a}_0, \btext{a}_1, \btext{a}_3] + [\btext{a}_0, \btext{a}_2, \btext{a}_3] - [\btext{a}_1, \btext{a}_2, \btext{a}_3]
\]
is a 2-cycle. That is, $c \in Z_2(K)$.
\end{example}

\begin{definition}[Group of $p$-Boundaries]\label{d:Group of p Boundaries}
Let $K$ be a simplicial complex. The group of $p$-boundaries, denoted by $B_p(K)$, is defined as the image of the boundary operator $\partial_{p+1} \colon C_{p+1}(K) \to C_p(K)$. That is, 
\[B_p(K) = \mrm{im}(\partial_{P + 1}) = \left\{ \partial_{p+1}(c) \mid c \in C_{p+1}(K) \right\}
\]
We can say that $B_p(K)$ consists of all $p$-chains that are boundaries of some $(p+1)$-chain.
\end{definition}

\begin{example}[1-Boundary on a Triangle]
Let $K$ be a simplicial complex representing a filled triangle (a 2-simplex) with ordered vertex set $V = \{ \bm{a}_0, \bm{a}_1, \bm{a}_2 \}$. The 1-simplices (oriented edges) of $K$ are given by
\[
\sigma_1 = [\bm{a}_0, \bm{a}_1], \quad
\sigma_2 = [\bm{a}_1, \bm{a}_2], \quad
\sigma_3 = [\bm{a}_2, \bm{a}_0].
\]
Let $\tau = [\bm{a}_0, \bm{a}_1, \bm{a}_2]$ be the oriented 2-simplex (triangle) formed by these vertices. The boundary of $\tau$ under the boundary homomorphism $\partial_2 \colon C_2(K) \to C_1(K)$ is computed using the standard alternating sum
\[
\partial_2(\tau) = [\bm{a}_1, \bm{a}_2] - [\bm{a}_0, \bm{a}_2] + [\bm{a}_0, \bm{a}_1].
\]
Note that we orient each face consistently with the orientation of $\tau$. Observe that
\[
[\bm{a}_0, \bm{a}_2] = -[\bm{a}_2, \bm{a}_0],
\]
so we rewrite the expression using the given edge notation:
\[
\partial_2(\tau) = \sigma_2 - (-\sigma_3) + \sigma_1 = \sigma_2 + \sigma_3 + \sigma_1.
\]
Hence, the 1-boundary corresponding to the 2-simplex $\tau = [\bm{a}_0, \bm{a}_1, \bm{a}_2]$ is 
\[\partial_2([\bm{a}_0, \bm{a}_1, \bm{a}_2]) = \sigma_1 + \sigma_2 + \sigma_3\]
This 1-chain is a 1-boundary in $C_1(K)$, i.e., an element of $B_1(K)$. It represents the formal sum of the edges that bound the triangle. Finally, we can say that 
\[
B_1(K) = \left\{ n(\sigma_1 + \sigma_2 + \sigma_3) \colon  n \in \mathbb{Z} \right\}
\]
\end{example}

\begin{example}[2-Boundaries of a Tetrahedron]
Let $K$ be a simplicial complex representing a tetrahedron with the oriented vertex set $V = \{\bm{a}_0, \bm{a}_1, \bm{a}_2, \bm{a}_3$. The 2-simplices (oriented triangular faces) of $K$ are:
\[
\sigma_1 = [\bm{a}_0, \bm{a}_1, \bm{a}_2], \quad
\sigma_2 = [\bm{a}_0, \bm{a}_1, \bm{a}_3], \quad
\sigma_3 = [\bm{a}_0, \bm{a}_2, \bm{a}_3], \quad
\sigma_4 = [\bm{a}_1, \bm{a}_2, \bm{a}_3].
\]
The boundary operator $\partial_2$ applied to each 2-simplex is computed as follows:
\begin{gather*}
\partial_2(\sigma_1) = [\bm{a}_1, \bm{a}_2] - [\bm{a}_0, \bm{a}_2] + [\bm{a}_0, \bm{a}_1],\\
\partial_2(\sigma_2) = [\bm{a}_1, \bm{a}_3] - [\bm{a}_0, \bm{a}_3] + [\bm{a}_0, \bm{a}_1],\\
\partial_2(\sigma_3) = [\bm{a}_2, \bm{a}_3] - [\bm{a}_0, \bm{a}_3] + [\bm{a}_0, \bm{a}_2],\\
\partial_2(\sigma_4) = [\bm{a}_2, \bm{a}_3] - [\bm{a}_1, \bm{a}_3] + [\bm{a}_1, \bm{a}_2]
\end{gather*}
Thus, the 2-boundary group $B_2(K)$, which consists of all integer linear combinations of these boundaries, is given by:
\[
B_2(K) = \left\{
n_1 \partial_2(\sigma_1) + n_2 \partial_2(\sigma_2) + n_3 \partial_2(\sigma_3) + n_4 \partial_2(\sigma_4)
\colon n_1, n_2, n_3, n_4 \in \mathbb{Z}
\right\}
\]
\end{example}

\begin{theorem}[$B_p(K) \subseteq Z_p(K)$]\label{t:Group of p boundaries is a subgroup of the group of p cycles}
For any simplicial complex $K$, the group of $p$-boundaries is a subgroup of the group of $p$-cycles. That is, $B_p(K) \subseteq Z_p(K)$ for all $p \geq 0$.
\end{theorem}

\begin{proof}
Recall that for each $p \geq 0$, we have the boundary homomorphisms
\[
\partial_{p+1} \colon C_{p+1}(K) \to C_p(K), \quad \text{and} \quad
\partial_p \colon C_p(K) \to C_{p-1}(K)
\]
which satisfy the fundamental property $\partial_p \circ \partial_{p+1} = 0$. Let $b \in B_p(K)$. Then, by definition, there exists a $(p+1)$-chain $c \in C_{p+1}(K)$ such that
\[
b = \partial_{p+1}(c).
\]
Applying $\partial_p$ to both sides
\[
\partial_p(b) = \partial_p(\partial_{p+1}(c)) = (\partial_p \circ \partial_{p+1})(c) = 0.
\]
Thus, $b \in \ker(\partial_p) = Z_p(K)$, which shows that every $p$-boundary is a $p$-cycle. Therefore, $B_p(K) \subseteq Z_p(K)$.
\end{proof}

\begin{definition}[Homology Group]\label{d:Homology Group}
Let $K$ be a simplicial complex. The \textbf{$p$-th homology group} of $K$, denoted by $H_p(K)$, is the quotient group of the group of $p$-cycles by the group of $p$-boundaries. That is,
\[
H_p(K) = Z_p(K) / B_p(K)
\]
The elements of $H_p(K)$ are the equivalence classes of $p$-cycles modulo $p$-boundaries:
\[
H_p(K) = \left\{ z + B_p(K) \mid z \in Z_p(K) \right\}
\]
\end{definition}
\begin{question}
What is interpretation of different value of the homology group $H_{p}(K)$ of the simplicial complex $K$.   
\end{question}

\begin{answer}
The interpretation of the different values of the homology group $H_{p}(K)$ of a simplicial complex $K$ is described in the following table.
\begin{figure}[h!]
\centering
\renewcommand{\arraystretch}{1.8}
\begin{tabular}{|c|m{6.5cm}|m{5cm}|}
\hline
\multicolumn{1}{|c|}{\textbf{Homology Group}} & 
\multicolumn{1}{c|}{\textbf{Topological Interpretation}} & 
\multicolumn{1}{c|}{\textbf{Example}} \\
\hline
$H_0(K)$ & Counts the number of \textbf{connected components} of the simplicial complex $K$. & $H_0(K) \cong \mathbb{Z}^2$ means $K$ has two disconnected parts. \\
\hline
$H_1(K)$ & Counts the number of \textbf{independent loops or 1-dimensional holes} that are not boundaries of filled triangles. & $H_1(K) \cong \mathbb{Z}$ for the boundary of a triangle or square. \\
\hline
$H_2(K)$ & Counts the number of \textbf{2-dimensional voids or cavities}, like enclosed surfaces. & $H_2(K) \cong \mathbb{Z}$ for the surface of a sphere or torus. \\
\hline
$H_p(K)$ (General) & Detects \textbf{$p$-dimensional holes} that are not boundaries of $(p+1)$-simplices. Describes shape in dimension $p$. & Higher-dimensional torus: $H_3(K) \cong \mathbb{Z}$ indicates a 3D cavity in a 4D complex. \\
\hline
\end{tabular}
\caption{Interpretation of Homology Group $H_{p}(K)$}
\label{fig:Interpretation of Homology Group}
\end{figure}
\end{answer}

\begin{theorem}
Let $K$ be a simplicial complex. Then the $p$-th homology group
\[
H_p(K) = Z_p(K)/B_p(K)
\]
is a well-defined abelian group, where $Z_p(K)$ is the group of $p$-cycles and $B_p(K)$ is the group of $p$-boundaries.
\end{theorem}

\begin{proof}
Since the chain group $C_p(K)$ is an abelian group, and the boundary homomorphisms
\[
\partial_{p+1} \colon C_{p+1}(K) \to C_p(K), \quad \partial_p \colon C_p(K) \to C_{p-1}(K)
\]
satisfy $\partial_p \circ \partial_{p+1} = 0$, it follows that
\[
\text{im}(\partial_{p+1}) \subseteq \ker(\partial_p), \quad \text{i.e.,} \quad B_p(K) \subseteq Z_p(K).
\]
Hence, $B_p(K)$ is a subgroup of $Z_p(K)$. The quotient $Z_p(K)/B_p(K)$ is therefore a well-defined group, called the $p$-th homology group $H_p(K)$.
\end{proof}

\section{Computation of Homology Group of Simplicial Complex}\label{s:Computation of Homology Group of Simplicial Complex}

\begin{algorithm}[Computation of $p$-th Homology Group $H_p(K)$]
Let $K$ be a finite simplicial complex. This algorithm computes the $p$-th homology group $H_p(K) = \ker(\partial_p)/\operatorname{im}(\partial_{p+1})$.

\begin{enumerate}[label=\textbf{Step \arabic*.}, itemsep=1em]
    \item \textbf{List all simplices.}\\
    For each dimension $p = 0, 1, 2, \dots$, list all $p$-simplices $\sigma_i$ in $K$ with a fixed orientation.

    \item \textbf{Construct chain groups.} \\
    Define the $p$-chain group $C_p(K)$ as the free abelian group generated by $p$-simplices:
    \[
    C_p(K) = \mathbb{Z}\langle \sigma_1, \sigma_2, \dots, \sigma_n \rangle.
    \]
    
    \item \textbf{Define boundary maps.} \\
    Define the boundary homomorphism $\partial_p: C_p(K) \to C_{p-1}(K)$ and represent it as a matrix. Each entry depends on the incidence and orientation of simplices:
    \begin{itemize}
        \item $+1$ if the face occurs with positive orientation,
        \item $-1$ if with negative orientation,
        \item $0$ otherwise.
    \end{itemize}

    \item \textbf{Compute cycles and boundaries.}
    \begin{itemize}
        \item Compute the kernel: $Z_p(K) = \ker(\partial_p)$ (the group of $p$-cycles).
        \item Compute the image: $B_p(K) = \operatorname{im}(\partial_{p+1})$ (the group of $p$-boundaries).
    \end{itemize}
    Use row reduction or Smith normal form to determine generators.

    \item \textbf{Compute homology group.} \\
    Take the quotient:
    \[
    H_p(K) = \frac{Z_p(K)}{B_p(K)} = \frac{\ker(\partial_p)}{\operatorname{im}(\partial_{p+1})}.
    \]
    Express $H_p(K)$ as a direct sum of abelian groups:
    \[
    H_p(K) \cong \mathbb{Z}^{\beta_p} \oplus \mathbb{Z}_{d_1} \oplus \cdots \oplus \mathbb{Z}_{d_k},
    \]
    where $\beta_p$ is the $p$-th Betti number, and $\mathbb{Z}_{d_i}$ are torsion components.
\end{enumerate}
\end{algorithm}

Let us see the computation of homology groups for the boundary of a square using Smith Normal Form

\begin{example}[Homology Group of Boundary of Triangle]\label{eg:Homology Group of Boundary of Triangle}
Find the $1$-th homology group of the simplicial complex $K$ whose underlying space is boundary of triangle.  

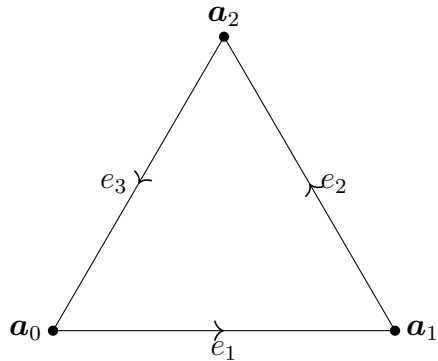
\begin{figure}[h!]
    \centering
    \begin{tikzpicture}[scale=4.5,
  edge/.style={postaction={decorate}, decoration={markings,
  mark=at position 0.5 with {\arrow[scale=1.5]{>}}}}] % larger arrows

% Define triangle vertices
\coordinate (A0) at (0,0);
\coordinate (A1) at (1,0);
\coordinate (A2) at (0.5,0.866); % height of equilateral triangle

% Draw edges
\draw[edge] (A0) -- node[below] {$e_1$} (A1);
\draw[edge] (A1) -- node[right] {$e_2$} (A2);
\draw[edge] (A2) -- node[left] {$e_3$} (A0);

% Draw filled circles at vertices
\foreach \pt in {A0, A1, A2}
  \fill (\pt) circle[radius=0.015cm];

% Label vertices
\node[left] at (A0) {$\bm{a}_0$};
\node[right] at (A1) {$\bm{a}_1$};
\node[above] at (A2) {$\bm{a}_2$};

% Label the complex
%\node at (-0.2, 0.4) {$K$};
\end{tikzpicture}
    \caption{Boundary of Triangle}
    \label{fig:Boundary of Triangle}
\end{figure}
Let $V = \{\bm{a}_0, \bm{a}_1,\bm{a}_2\}$ be the set of vertex of the simplicial complex $K$ whose underlying space is boundary of triangle. Then $K$ is 
\[K = \{\{\bm{a}_0\}, \{\bm{a}_1\}, \{\bm{a}_2\}, \{\bm{a}_0, \bm{a}_1\}, \{\bm{a}_1,\bm{a}_2\}, \{\bm{a}_2,\bm{a}_0\}\}\]
Therefore, the sets of oriented 0-simplices and oriented 1-simplices are respectively
\begin{gather*}
[\bm{a}_0], [\bm{a}_1], [\bm{a}_2] \\
[\bm{a}_0, \bm{a}_1], [\bm{a}_1, \bm{a}_2], [\bm{a}_2, \bm{a}_0]
\end{gather*}
We want to find $1$-th homology group $H_{1}(K)$ of the simplicial complex $K$. So by definition, 
\begin{equation}\label{eq:1 th homology group}
H_{1}(K) = \dfrac{Z_{1}(K)}{B_{1}(K)} = \dfrac{\ker(\partial_1)}{\mrm{im}(\partial_2)}
\end{equation}
where maps $\partial_1 \colon C_{1}(K) \to C_{0}(K)$ and $\partial_2 \colon C_{2}(K) \to C_{1}(K)$ are boundary maps on the group of 1-chains $C_{1}(K)$ and  group of 2-chains $C_{2}(K)$ respectively. 

Now first we find the groups of chain $C_{0}(K),  C_{1}(K)$ and $ C_{2}(K)$. By definition of group of chain,
\begin{equation}\label{eq:0 chain group}
C_{0}(K) = \left \{c = \sum_{i = 0}^{2}n_{i}[\bm{a}_{i}] \colon n_{i} \in \mbb{Z} \right \} = \mbb{Z}\langle [\bm{a}_0], [\bm{a}_1], [\bm{a}_2] \rangle \cong \mbb{Z}^3
\end{equation}
Where $\mbb{Z}\langle [\bm{a}_0], [\bm{a}_1], [\bm{a}_2] \rangle$ is free abelian group on set $\{[\bm{a}_0], [\bm{a}_1], [\bm{a}_2]\}$ in which every elements is a formal linear combination of these generators $[\bm{a}_0], [\bm{a}_1], [\bm{a}_2]$ with integer coefficients. The $\mbb{Z}^{3} = \{(n_0, n_1, n_2) \colon n_i \in \mbb{Z}\}$ is abelian group. We have a map $\phi \colon C_{0}(K) \to \mbb{Z}^3$ defined by $\phi(c) = (n_0, n_1, n_2)$ for all $c \in C_{0}(K)$ which is homomorphism, therefore $C_{0}(K) \cong \mbb{Z}^3$. Similarly, $C_{1}(K) \cong \mbb{Z}^3$. Since 2-simplices does not exists in $K$, therefore $C_{2}(K) = 0$ is trivial. Finally,
\begin{equation}\label{eq:chain groups}
C_{0}(K) \cong \mbb{Z}^3, C_{1}(K) \cong \mbb{Z}^3,  
C_{2}(K) = 0  
\end{equation}
The boundary map $\partial_1 \colon C_1(K) \to C_0(K)$ is defined by:
\begin{align*}
\partial_1[\bm{a}_0, \bm{a}_1] &= [\bm{a}_1] - [\bm{a}_0] \\
\partial_1[\bm{a}_1, \bm{a}_2] &= [\bm{a}_2] - [\bm{a}_1] \\
\partial_1[\bm{a}_2, \bm{a}_0] &= [\bm{a}_0] - [\bm{a}_2]
\end{align*}
In matrix form, with respect to the ordered bases 
\[
C_1(K): ([\bm{a}_0, \bm{a}_1], [\bm{a}_1, \bm{a}_2], [\bm{a}_2, \bm{a}_0]) \\
C_0(K): ([\bm{a}_0], [\bm{a}_1], [\bm{a}_2])
\]
we have:
\[
[\partial_1] =
\begin{bmatrix}
-1 &  0 &  1 \\
 1 & -1 &  0 \\
 0 &  1 & -1
\end{bmatrix}
\]
To compute $\ker(\partial_1)$, let $x = (x_1, x_2, x_3)^T \in \mathbb{Z}^3$ be a 1-chain:
\[
x = x_1[\bm{a}_0, \bm{a}_1] + x_2[\bm{a}_1, \bm{a}_2] + x_3[\bm{a}_2, \bm{a}_0]
\]
Then $\partial_1(x) = 0$ gives the system:
\begin{align*}
-x_1 + x_3 &= 0 \\
x_1 - x_2 &= 0 \\
x_2 - x_3 &= 0
\end{align*}
which implies $x_1 = x_2 = x_3$. Thus,
\[
\ker(\partial_1) = \mathbb{Z} \langle [\bm{a}_0, \bm{a}_1] + [\bm{a}_1, \bm{a}_2] + [\bm{a}_2, \bm{a}_0] \rangle \cong \mathbb{Z}
\]
Since $C_2(K) = 0$, we have $\mathrm{im}(\partial_2) = 0$, so
\[
H_1(K) = \dfrac{\ker(\partial_1)}{\mathrm{im}(\partial_2)} \cong \ker(\partial_1) \cong \mathbb{Z}
\]
The first homology group of the simplicial complex $K$ is $H_1(K) \cong \mathbb{Z}$.
\end{example}

\begin{example}[Homology Group of Boundary of Square]\label{eg:Homology Group of Boundary of Square}
Find the $1$-th homology group of the simplicial complex $K$ whose underlying space is boundary of square.   

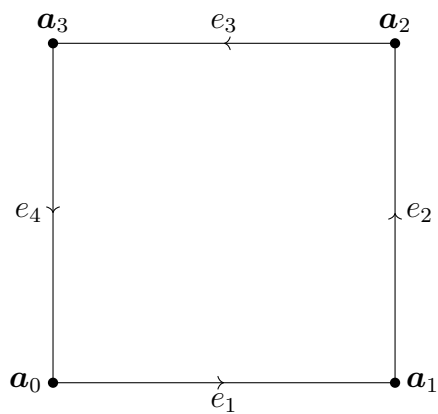
\begin{figure}[h!]
\centering
\begin{tikzpicture}[scale=4.5,
  edge/.style={postaction={decorate}, decoration={markings,
  mark=at position 0.5 with {\arrow{>}}}}]

% Define vertices
\coordinate (A0) at (0,0);
\coordinate (A1) at (1,0);
\coordinate (A2) at (1,1);
\coordinate (A3) at (0,1);

% Draw square edges with middle arrows
\draw[edge] (A0) -- node[below] {$e_1$} (A1);
\draw[edge] (A1) -- node[right] {$e_2$} (A2);
\draw[edge] (A2) -- node[above] {$e_3$} (A3);
\draw[edge] (A3) -- node[left] {$e_4$} (A0);

% Draw diagonal edge
%\draw[edge] (A0) -- node[below right] {$e_5$} (A2);

% Label vertices
\node[left] at (A0) {$\bm{a}_0$};
\node[right] at (A1) {$\bm{a}_1$};
\node[above] at (A2) {$\bm{a}_2$};
\node[above] at (A3) {$\bm{a}_3$};

% Draw filled circles at vertices
\foreach \pt in {A0, A1, A2, A3}
  \fill (\pt) circle[radius=0.015cm];

% Smaller and rotated curved arrows for sigma_1 and sigma_2
%\draw[<-, thick] (0.58,0.12) arc[start angle=200,end angle=30,radius=0.18cm];
%\node at (0.38,0.12) {\small $\sigma_1$};

%\draw[<-, thick] (0.22,0.48)arc[start angle=200,end angle=30,radius=0.18cm];
%\node at (0.38,0.88) {\small $\sigma_2$};
% Label the complex
%\node at (-0.12, 0.5) {$L$};
\end{tikzpicture}
\caption{Boundary of Square}
\label{fig:Boundary of Square}
\end{figure}

Let $K$ be the simplicial complex whose underlying space is the boundary of a square. The vertex set is $V = \{\bm{a}_0, \bm{a}_1, \bm{a}_2, \bm{a}_3\}$ and the simplicial complex is given by
\[
K = \left\{
\{\bm{a}_0\}, \{\bm{a}_1\}, \{\bm{a}_2\}, \{\bm{a}_3\},
\{\bm{a}_0, \bm{a}_1\}, \{\bm{a}_1, \bm{a}_2\}, \{\bm{a}_2, \bm{a}_3\}, \{\bm{a}_3, \bm{a}_0\}
\right\}
\]
The oriented simplices are
\begin{itemize}
    \item 0-simplices: $[\bm{a}_0], [\bm{a}_1], [\bm{a}_2], [\bm{a}_3]$
    \item 1-simplices: $[\bm{a}_0, \bm{a}_1], [\bm{a}_1, \bm{a}_2], [\bm{a}_2, \bm{a}_3], [\bm{a}_3, \bm{a}_0]$
\end{itemize}
Since there are no 2-simplices, $C_2(K) = 0$. Now we find the chain groups as follows:
\[
\begin{aligned}
C_0(K) &= \mathbb{Z}\langle [\bm{a}_0], [\bm{a}_1], [\bm{a}_2], [\bm{a}_3] \rangle \cong \mathbb{Z}^4 \\
C_1(K) &= \mathbb{Z}\langle [\bm{a}_0, \bm{a}_1], [\bm{a}_1, \bm{a}_2], [\bm{a}_2, \bm{a}_3], [\bm{a}_3, \bm{a}_0] \rangle \cong \mathbb{Z}^4 \\
C_2(K) &= 0
\end{aligned}
\]
The Boundary Map $\partial_1 \colon C_1(K) \to C_0(K)$ as follows:
\[
\begin{aligned}
\partial_1[\bm{a}_0, \bm{a}_1] &= [\bm{a}_1] - [\bm{a}_0] \\
\partial_1[\bm{a}_1, \bm{a}_2] &= [\bm{a}_2] - [\bm{a}_1] \\
\partial_1[\bm{a}_2, \bm{a}_3] &= [\bm{a}_3] - [\bm{a}_2] \\
\partial_1[\bm{a}_3, \bm{a}_0] &= [\bm{a}_0] - [\bm{a}_3]
\end{aligned}
\]
The matrix representation of $\partial_1$ with respect to the standard bases is
\[
[\partial_1] = 
\begin{bmatrix}
-1 & 0 & 0 & 1 \\
1 & -1 & 0 & 0 \\
0 & 1 & -1 & 0 \\
0 & 0 & 1 & -1
\end{bmatrix}
\]
Now we compute $\ker(\partial_1)$. Let $\bm{x} = (x_1, x_2, x_3, x_4)^T \in C_1(K)$. Then
\[
\partial_1(\bm{x}) = 0 \Longleftrightarrow
\begin{cases}
-x_1 + x_4 = 0 \\
x_1 - x_2 = 0 \\
x_2 - x_3 = 0 \\
x_3 - x_4 = 0
\end{cases}
\Rightarrow x_1 = x_2 = x_3 = x_4
\]
Hence,
\[
\ker(\partial_1) = \mathbb{Z}\left( [\bm{a}_0, \bm{a}_1] + [\bm{a}_1, \bm{a}_2] + [\bm{a}_2, \bm{a}_3] + [\bm{a}_3, \bm{a}_0] \right) \cong \mathbb{Z}
\]
Now we compute $H_1(K)$. Since $C_2(K) = 0$, we have $\operatorname{im}(\partial_2) = 0$. Therefore,
\[
H_1(K) = \frac{\ker(\partial_1)}{\operatorname{im}(\partial_2)} = \ker(\partial_1) \cong \mathbb{Z}
\]
So $H_1(K) \cong \mathbb{Z}$.
\end{example}

\begin{example}[Homology Group of Triangulated Boundary of Square]\label{eg:Homology Group of Triangulated Boundary of Square}
Find the $1$-th homology group of the simplicial complex $K$ whose underlying space is triangulated boundary of square.     

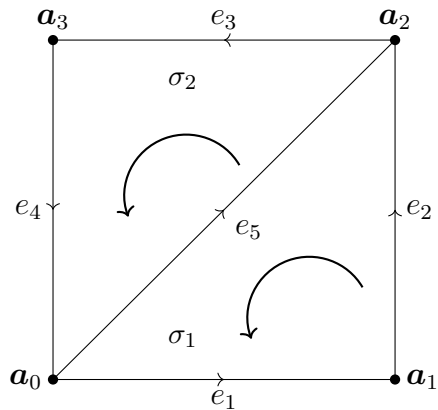
\begin{figure}[h!]
\centering
\begin{tikzpicture}[scale=4.5,
  edge/.style={postaction={decorate}, decoration={markings,
  mark=at position 0.5 with {\arrow{>}}}}]

% Define vertices
\coordinate (A0) at (0,0);
\coordinate (A1) at (1,0);
\coordinate (A2) at (1,1);
\coordinate (A3) at (0,1);

% Draw square edges with middle arrows
\draw[edge] (A0) -- node[below] {$e_1$} (A1);
\draw[edge] (A1) -- node[right] {$e_2$} (A2);
\draw[edge] (A2) -- node[above] {$e_3$} (A3);
\draw[edge] (A3) -- node[left] {$e_4$} (A0);

% Draw diagonal edge
\draw[edge] (A0) -- node[below right] {$e_5$} (A2);

% Label vertices
\node[left] at (A0) {$\bm{a}_0$};
\node[right] at (A1) {$\bm{a}_1$};
\node[above] at (A2) {$\bm{a}_2$};
\node[above] at (A3) {$\bm{a}_3$};

% Draw filled circles at vertices
\foreach \pt in {A0, A1, A2, A3}
  \fill (\pt) circle[radius=0.015cm];

% Smaller and rotated curved arrows for sigma_1 and sigma_2
\draw[<-, thick] (0.58,0.12) arc[start angle=200,end angle=30,radius=0.18cm];
\node at (0.38,0.12) {\small $\sigma_1$};

\draw[<-, thick] (0.22,0.48)arc[start angle=200,end angle=30,radius=0.18cm];
\node at (0.38,0.88) {\small $\sigma_2$};
% Label the complex
%\node at (-0.12, 0.5) {$L$};
\end{tikzpicture}
\caption{Triangulated Boundary of Square}
\label{fig:Triangulated Boundary of Square}
\end{figure}
Let $V = \{\bm{a}_0, \bm{a}_1,\bm{a}_2, \bm{a}_3\}$ be the set of vertex of the simplicial complex $K$ whose underlying space is a square subdivided into two triangles by a diagonal. We assume the square is oriented so that
\begin{itemize}
    \item The \emph{boundary edges} are:
    \[
    e_1 = [\bm{a}_0, \bm{a}_1],\quad e_2 = [\bm{a}_1, \bm{a}_2],\quad e_3 = [\bm{a}_2, \bm{a}_3],\quad e_4 = [\bm{a}_3, \bm{a}_0].
    \]
    \item The \emph{diagonal} is:
    \[
    e_5 = [\bm{a}_0, \bm{a}_2].
    \]
    \item The two \emph{2-simplices} (triangles) are:
    \[
    \sigma_1 = [\bm{a}_0, \bm{a}_1, \bm{a}_2], \quad
    \sigma_2 = [\bm{a}_0, \bm{a}_2, \bm{a}_3].
    \]
\end{itemize}
Thus, the simplicial complex is:
\begin{multline*}
K = \{
\{\bm{a}_0\}, \{\bm{a}_1\}, \{\bm{a}_2\}, \{\bm{a}_3\}, \{\bm{a}_0, \bm{a}_1\}, \{\bm{a}_1,\bm{a}_2\},  \{\bm{a}_2,\bm{a}_3\}, \{\bm{a}_3,\bm{a}_0\}, \{\bm{a}_0,\bm{a}_2\},\\
\{\bm{a}_0, \bm{a}_1, \bm{a}_2\}, \{\bm{a}_0, \bm{a}_2, \bm{a}_3\}
\}   
\end{multline*}
\textbf{Chain Groups:}The chain groups are as follows:
\[
\begin{aligned}
C_2(K) &= \mathbb{Z}\langle \sigma_1, \sigma_2 \rangle \cong \mathbb{Z}^2, \\
C_1(K) &= \mathbb{Z}\langle e_1, e_2, e_3, e_4, e_5 \rangle \cong \mathbb{Z}^5, \\
C_0(K) &= \mathbb{Z}\langle [\bm{a}_0], [\bm{a}_1], [\bm{a}_2], [\bm{a}_3] \rangle \cong \mathbb{Z}^4.
\end{aligned}
\]
\textbf{Boundary Maps:}Now we define boundary maps $\partial_2$ and $\partial_1$ on the $C_{2}(K)$ and $C_{1}(K)$ respectively. 

\paragraph{The map \(\partial_2 \colon C_2(K) \to C_1(K)\):} 

Recall that for a 2-simplex \([v_0,v_1,v_2]\) the boundary is
\[
\partial_2([v_0,v_1,v_2]) = [v_1,v_2] - [v_0,v_2] + [v_0,v_1].
\]
Thus,
\[
\begin{aligned}
\partial_2(\sigma_1) &= [\bm{a}_1,\bm{a}_2] - [\bm{a}_0,\bm{a}_2] + [\bm{a}_0,\bm{a}_1] = e_2 - e_5 + e_1, \\
\partial_2(\sigma_2) &= [\bm{a}_2,\bm{a}_3] - [\bm{a}_0,\bm{a}_3] + [\bm{a}_0,\bm{a}_2].
\end{aligned}
\]
Since the edge \([\bm{a}_0, \bm{a}_3]\) is not in our list of oriented 1-simplices, we express it in terms of the chosen orientations:
\[
[\bm{a}_0,\bm{a}_3] = -[\bm{a}_3,\bm{a}_0] = -e_4.
\]
Thus, 
\[
\partial_2(\sigma_2) = [\bm{a}_2,\bm{a}_3] - (-e_4) + [\bm{a}_0,\bm{a}_2] = e_3 + e_4 + e_5.
\]
It is common to choose an orientation on the 2-simplices so that the boundaries match. Here, if we reverse the orientation on \(\sigma_2\) (or reinterpret the sign on \(e_5\)) we can obtain consistency. For our purposes, note that the images are:
\[
\operatorname{im}(\partial_2) = \operatorname{span}_{\mathbb{Z}}\{\, e_1 + e_2 - e_5,\; e_3 + e_4 + e_5 \,\}.
\]

\paragraph{The map \(\partial_1 \colon C_1(K) \to C_0(K)\):}

For a 1-simplex \([v_i,v_j]\) we have
\[
\partial_1([v_i,v_j]) = [v_j] - [v_i].
\]
Thus,
\[
\begin{aligned}
\partial_1(e_1) &= [\bm{a}_1] - [\bm{a}_0],\\[1mm]
\partial_1(e_2) &= [\bm{a}_2] - [\bm{a}_1],\\[1mm]
\partial_1(e_3) &= [\bm{a}_3] - [\bm{a}_2],\\[1mm]
\partial_1(e_4) &= [\bm{a}_0] - [\bm{a}_3],\\[1mm]
\partial_1(e_5) &= [\bm{a}_2] - [\bm{a}_0].
\end{aligned}
\]
A matrix representation of \(\partial_1\) (with basis \( e_1,e_2,e_3,e_4,e_5\) for \(C_1\) and \([\bm{a}_0],[\bm{a}_1],[\bm{a}_2],[\bm{a}_3]\) for \(C_0\)) is:
\[
[\partial_1] =
\begin{bmatrix}
-1 &  0 &  0 &  1 & -1 \\
 1 & -1 &  0 &  0 &  0 \\
 0 &  1 & -1 &  0 &  1 \\
 0 &  0 &  1 & -1 &  0 
\end{bmatrix}.
\]

\textbf{Computing \(H_1(K)\)}

The first homology group is defined as
\[
H_1(K) = \frac{\ker(\partial_1)}{\operatorname{im}(\partial_2)}.
\]

\paragraph{Step 1: The Rank of \(\ker(\partial_1)\)}
The group \(C_1(K)\) has rank 5. By analyzing the matrix of \(\partial_1\) (or by appropriate row reduction), one may show that the rank of \(\operatorname{im}(\partial_1)\) is 3. Then, by the rank-nullity theorem,
\[
\operatorname{rank}(\ker(\partial_1)) = 5 - 3 = 2.
\]

\paragraph{Step 2: The Rank of \(\operatorname{im}(\partial_2)\)}
Since we have two 2-simplices and the images
\[
\partial_2(\sigma_1) = e_1 + e_2 - e_5, \quad \partial_2(\sigma_2) = e_3 + e_4 + e_5,
\]
one can check that these two elements are independent in \(C_1(K)\). Hence, 
\[
\operatorname{rank}(\operatorname{im}(\partial_2)) = 2.
\]

\paragraph{Step 3: Conclusion for \(H_1(K)\)}
Since \(\ker(\partial_1)\) is of rank 2 and \(\operatorname{im}(\partial_2)\) is a subgroup of rank 2 inside \(\ker(\partial_1)\), we have
\[
H_1(K) = \frac{\ker(\partial_1)}{\operatorname{im}(\partial_2)} \cong 0.
\]
Therefore, $H_1(K) = 0$.
\end{example}

\begin{observation}
From above three examples we can see that the homology group $H_{p}(K)$ of simplicial complex $K$ depends only on the underlying space $|K|$.  
\end{observation}
Hence, an important question arises as follows.  
\begin{question}
What happens to the calculation of homology groups of simplicial complexes if we increase the number of simplices in it.  
\end{question}

\begin{answer}
If we increase the number of simplices in a simplicial complex, the calculation of homology groups becomes:
\begin{enumerate}
    
\item More complex computationally: There are more chains and boundaries to consider, increasing the size of chain groups and the complexity of boundary maps.

\item  More refined topologically: The homology groups may better capture the topological structure, possibly detecting additional features (like holes) that were not represented with fewer simplices.

\item But the actual homology groups might not change if the added simplices do not alter the topology of the space (e.g., adding subdivisions or extra simplices that fill in nothing new).
\end{enumerate}
\end{answer}

So, increasing simplices refines the combinatorial structure but doesn't necessarily change the topological features the homology detects. Moreover it is not easy to compute the various components like groups of cycles $Z_{p}(K)$ and $B_{p}(K)$ while we are increasing the simplices and eventually it is difficult to compute the homology group in this case. Therefore, let us directly look at a short cut method for calculating the homology group which does not involve the groups of cycles $Z_{p}(K)$ and $B_{p}(K)$. Let us first look at some important definitions that are directly involved in the short cut method of calculating homology groups.

\begin{definition}[Chain carried by Subcomplex]\label{d:Chain carried by Subcomplex}
A chain $c \in C_p(K)$ is carried by a subcomplex $L \subseteq K$ if $c(\sigma) = 0$ for every $p$-simplex $\sigma \notin L$. In simpler terms, $c$ is carried by $L$ if it assigns nonzero coefficients only to simplices that are in $L$ and assigns zero to everything outside of $L$.
\end{definition}

\begin{example}
Let simplicial complex representing the boundary of a triangle with verttex set $V = \{\bm{a}_0, \bm{a}_1, \bm{a}_2\}$ is 
\[K = \{\{\bm{a}_0\}, \{\bm{a}_0\}, \{\bm{a}_0\}, \{\bm{a}_0, \bm{a}_1\}, \{\bm{a}_1, \bm{a}_2\}, \{\bm{a}_2, \bm{a}_0\}\}\]
Let subcomplex of $K$ consisting of only two edges and the three vertices is 
\[L = \{\{\bm{a}_0\}, \{\bm{a}_0\}, \{\bm{a}_0\}, \{\bm{a}_0, \bm{a}_1\}, \{\bm{a}_1, \bm{a}_2\}\}\]
Now, consider the 1-chain $c \in C_{1}(K)$ such that
\[
c = 2[\bm{a}_0, \bm{a}_1] + 3[\bm{a}_1, \bm{a}_2] + 0[\bm{a}_2, \bm{a}_0]
\]
We see that
\begin{itemize}
\item The edges $[\bm{a}_0, \bm{a}_1]$ and $[\bm{a}_1, \bm{a}_2]$ appear with nonzero coefficients.
\item The edge $[\bm{a}_2, \bm{a}_0]$, which is not in $L$, appears with coefficient zero.
\end{itemize}
Therefore, chain $c$ is carried by the subcomplex $L$, because it assigns nonzero values only to simplices that belong to $L$, and zero to the simplex $[\bm{a}_2, \bm{a}_0]$, which is not in $L$.
\end{example}

\begin{definition}[Homologous Chains]\label{d:Homologous Chains}
Two $p$-chains $c$ and $c'$ in $C_{p}(K)$ are said to be homologous if their difference is a boundary; that is, 
\[c - c' = \partial_{p + 1}(d)\]
for some $(p + 1)$ chain $d$. Being homologous means $c$ and $c'$ represent the same element in the homology group $H_p(K)$.   
\end{definition}

\begin{remark}
In particular, if $c = \partial_{p + 1}(d)$, we say that $c$ is homologous to zero, or simply $c$ is bounds. In this case we say that $c$ is boundary of a $(p + 1)$-simplex $d$.
\end{remark}

\begin{example}
Let simplicial complex representing a filled triangle with vertex set $V = \{\bm{a}_0, \bm{a}_1, \bm{a}_2\}$ is 
\[K = \{\{\bm{a}_0\}, \{\bm{a}_0\}, \{\bm{a}_0\}, \{\bm{a}_0, \bm{a}_1\}, \{\bm{a}_1, \bm{a}_2\}, \{\bm{a}_2, \bm{a}_0\}, \{\bm{a}_0, \bm{a}_1, \bm{a}_2\}\} \]
Define a 2-chain $d = [\bm{a}_0, \bm{a}_1, \bm{a}_2] \in C_2(K)$. Then the boundary of $d$ is
\[\partial_2(d) = [\bm{a}_1, \bm{a}_2] - [\bm{a}_0, \bm{a}_2] + [\bm{a}_0, \bm{a}_1]\]
Now consider two 1-chains $c$ and $c'$ in $C_{1}(K)$ such that 
\[c = [\bm{a}_1, \bm{a}_2] - [\bm{a}_0, \bm{a}_2] + [\bm{a}_0, \bm{a}_1], \qquad c' = 0\]
Then we have
\[c - c' = c = \partial_2(d)\]
Therefore, $c$ and $c'$ are homologous 1-chains because their difference is the boundary of the 2-chain $d$.
\end{example}

\begin{example}[Non-zero Homologous 1-Chains]\label{eg:Non-zero Homologous 1-Chains}
Let simplicial complex representing a filled triangle with vertex set $V = \{\bm{a}_0, \bm{a}_1, \bm{a}_2\}$ is 
\[K = \{\{\bm{a}_0\}, \{\bm{a}_0\}, \{\bm{a}_0\}, \{\bm{a}_0, \bm{a}_1\}, \{\bm{a}_1, \bm{a}_2\}, \{\bm{a}_2, \bm{a}_0\}, \{\bm{a}_0, \bm{a}_1, \bm{a}_2\}\} \]
Define a 2-chain $d = [\bm{a}_0, \bm{a}_1, \bm{a}_2] \in C_2(K)$. Then the boundary of $d$ is
\[\partial_2(d) = [\bm{a}_1, \bm{a}_2] - [\bm{a}_0, \bm{a}_2] + [\bm{a}_0, \bm{a}_1]\]
Now consider two 1-chains $c$ and $c'$ in $C_{1}(K)$ such that 
\[c = [\bm{a}_1, \bm{a}_2] + [\bm{a}_0, \bm{a}_1], \qquad c' = [\bm{a}_0, \bm{a}_2]\]
Then we have
\[c - c' = [\bm{a}_1,\bm{a}_2] + [\bm{a}_0,\bm{a}_1] - [\bm{a}_0,\bm{a}_2] = \partial_2(d)\]
Therefore, $c$ and $c'$ are homologous 1-chains because their difference is the boundary of the 2-chain $d$.
\end{example}

\begin{example}[Chain Homologous to Zero]\label{eg:Chain Homologous to Zero}
Let simplicial complex representing a filled triangle with vertex set $V = \{\bm{a}_0, \bm{a}_1, \bm{a}_2\}$ is 
\[K = \{\{\bm{a}_0\}, \{\bm{a}_0\}, \{\bm{a}_0\}, \{\bm{a}_0, \bm{a}_1\}, \{\bm{a}_1, \bm{a}_2\}, \{\bm{a}_2, \bm{a}_0\}, \{\bm{a}_0, \bm{a}_1, \bm{a}_2\}\} \]
Now define the 1-chain $c \in C_{1}(K)$
\[
c = [\bm{a}_1, \bm{a}_2] - [\bm{a}_0, \bm{a}_2] + [\bm{a}_0, \bm{a}_1]
\]
Then clearly, $c = \partial_2([\bm{a}_0, \bm{a}_1, \bm{a}_2])$. Thus, $c$ is a boundary and therefore is homologous to zero in the homology group $H_1(K)$. This means that the cycle formed by the three edges does not represent a hole; it actually bounds the 2-simplex (the triangle face), and hence is trivial in homology.
\end{example}

\begin{comment}

\begin{example}[Computation of Homology Group by Short Cut Method]\label{eg:Computation of Homology Group by Short Cut Method}
Find the $1$-th homology group of the simplicial complex $K$ whose underlying space is triangulated boundary of square as given in figure.

Let $V = \{\bm{a}_0, \bm{a}_1,\bm{a}_2, \bm{a}_3\}$ be the set of vertex of the simplicial complex $K$ whose underlying space is a square subdivided into four triangles by both diagonals. We assume the square is oriented so that
\begin{itemize}
\item The \emph{boundary edges} are:
\begin{gather*}
e_1 = [\bm{a}_0, \bm{a}_4],\quad e_2 = [\bm{a}_1, \bm{a}_4],\quad e_3 = [\bm{a}_2, \bm{a}_4],\quad e_4 = [\bm{a}_3, \bm{a}_4] \\
e_5 = [\bm{a}_0, \bm{a}_1],\quad e_6 = [\bm{a}_1, \bm{a}_2],\quad e_7 = [\bm{a}_2, \bm{a}_3],\quad e_8 = [\bm{a}_3, \bm{a}_0] 
\end{gather*}
\item The two \emph{2-simplices} (triangles) are:
\begin{gather*}
\sigma_1 = [\bm{a}_0, \bm{a}_1, \bm{a}_4], \quad
    \sigma_2 = [\bm{a}_1, \bm{a}_2, \bm{a}_4] \\
\sigma_3 = [\bm{a}_2, \bm{a}_3, \bm{a}_4], \quad
    \sigma_4 = [\bm{a}_3, \bm{a}_0, \bm{a}_4].
\end{gather*}
\end{itemize}
\end{example}
\end{comment}

\section{Conclusion}

This article developed simplicial homology for finite simplicial complexes from first principles, emphasizing how orientations, chain groups, and boundary operators assemble into a chain complex whose homology groups encode robust geometric information. The alternating–sign boundary formula and the identity \(\partial_{p-1}\circ \partial_p=0\) were verified on generators and extended linearly, clarifying the role of induced orientations and the cancellation of shared faces. Within this framework, cycles and boundaries were identified as kernels and images of boundary maps, yielding homology groups \(H_p(K)=\ker \partial_p / \operatorname{im}\partial_{p+1}\) that detect connected components, loops, and higher-dimensional voids.

A computation-oriented viewpoint was adopted throughout: fixing orientations, forming boundary matrices, and computing ranks of kernels and images provide a concrete route to \(H_p\). Low-dimensional examples illustrated typical sign conventions and demonstrated when \(H_1\) appears or vanishes, connecting algebraic outcomes to geometric features. The universal mapping property of chain groups further highlighted the linear nature of the constructions and the clarity gained by working on generators before extending by linearity.

Beyond the specific triangulations used in examples, the key message is robustness: while subdivisions alter combinatorics, simplicial homology remains invariant, reflecting properties of the underlying space rather than the chosen complex. Natural extensions include cohomology and the cup product, relative and reduced theories, and computational directions such as algorithmic homology and persistent homology. These avenues continue the central theme of the paper—translating geometry into computable algebra while maintaining clear sign conventions and a transparent linear-algebraic workflow.

For further reading, see the following literature: Hatcher’s Algebraic Topology provides a comprehensive, example‑rich foundation for homology and related invariants (Hatcher, 2002) \citep{Hatcher2002AlgebraicTopology}. Munkres’ Elements of Algebraic Topology offers a calculation‑focused pathway from chain groups and boundary maps to explicit homology computations (Munkres, 1996) \citep{Munkres1996ElementsAlgebraicTopology}. Ghrist’s Elementary Applied Topology supplies intuitive, application‑driven treatments that connect homological ideas to data and dynamics (Ghrist, 2014) \citep{Ghrist2014ElementaryAppliedTopology}. Rotman’s Advanced Modern Algebra develops the algebraic background—groups, modules, and exactness—underpinning chain complexes and universal properties (Rotman, 2002) \citep{Rotman2002AdvancedModernAlgebra}.

%***************************************************************************************************************************
%                                           Bibliography
%***************************************************************************************************************************
\bibliographystyle{unsrtnat}
\bibliography{SimplicialHomologyGroups}
\end{document}